\documentclass{article}
\usepackage{graphicx} 
\newcommand{\footremember}[2]{%
    \footnote{#2}
    \newcounter{#1}
    \setcounter{#1}{\value{footnote}}%
}
\title{Infinite Horizon Mean-Field Terminal Value Problem\thanks{Research of first and second authors was supported in part by the AFOSR Grant FA9550-24-1-0152}}
\author{%
  U\u{g}ur Ayd{\i}n\footremember{ua}{Department of Electrical and Computer Engineering and the Coordinated Science Laboratory, University of Illinois Urbana-Champaign, Urbana, IL, 61801, USA, uaydin2@illinois.edu}%
  \and Tamer Ba\c{s}ar\footremember{tb}{Department of Electrical and Computer Engineering and the Coordinated Science Laboratory, University of Illinois Urbana-Champaign, Urbana, IL, 61801, USA, basar1@illinois.edu}%
  \and Naci Saldi\footremember{ns}{Department of Mathematics, Bilkent University, \c{C}ankaya, Ankara, 06800, TURKEY, naci.saldi@bilkent.edu.tr}%
}
\date{}
\usepackage[a4paper]{geometry}
\usepackage[english]{babel}
\usepackage{a4wide}
\usepackage{indentfirst}
\usepackage{amsthm}
\usepackage{amsmath}
\usepackage{a4wide}
\usepackage{amstext}
\usepackage{amssymb}
\usepackage{tikz-cd}
\usepackage{float}
\usepackage[most]{tcolorbox}

\newtheorem{definition}{Definition}
\newtheorem{assumption}{Assumption}
\newtheorem{lemma}{Lemma}
\newtheorem{remark}{Remark}
\newtheorem{theorem}{Theorem}

\newtheorem{proposition}{Proposition}
\newtheorem{convention}{Convention}

\AtBeginEnvironment{algorithm}{\SetInd{.5em}{1em}}
\usepackage[boxruled]{algorithm2e}
\usepackage{algorithmic}

\usepackage{biblatex}
\bibliography{whole}
\usepackage{mathrsfs}
\usepackage{hyperref}
\usepackage{csquotes}
\begin{document}

\maketitle

\begin{abstract}
We introduce a class of finite- and infinite-horizon discrete-time control problems for studying macroscale systems in discrete time, which we refer to as mean-field terminal value problems (MFTVPs). An MFTVP takes a terminal \(Q\)-function as input and, starting from this terminal \(Q\)-function, seeks an indefinite backward evolution of optimal policies and state measures, as required for an infinite-horizon non-stationary mean-field equilibrium (MFE) of a discrete-time mean-field game (MFG). We study the existence and uniqueness properties of this problem.
As an application, using the uniqueness properties of MFTVPs, we show that, in the average-cost setting (and in the discounted setting for sufficiently small discount factors), MFTVPs can be used as an intermediate step to develop approximation schemes between finite-horizon MFEs and average-cost MFEs (respectively, stationary MFEs). In contrast to approaches that use infinite-horizon MFGs as intermediate problems, MFTVPs can be employed when the discount factor is sufficiently small (or, in the average-cost setting, when the drift factor is sufficiently small), together with a regularizer, to approximate a stationary MFE. Assuming uniqueness of the solution of a MFTVP, the restrictions we impose on the Lipschitz parameters of the system are weaker than the contraction conditions available for regularized stationary MFGs, and our approximation scheme remains applicable even in the presence of multiple stationary MFEs.
\end{abstract}

\begin{tcolorbox}[
    colback=gray!5,
    colframe=gray!60,
    boxrule=0.5pt,
    arc=2pt,
    left=6pt,
    right=6pt,
    top=5pt,
    bottom=5pt,
    title={Disclosure of AI Assistance},
    fonttitle=\bfseries
]
The first author used ChatGPT 6-Astra for language editing and to
assist in constructing the numerical examples presented in this
manuscript.
\end{tcolorbox}

\section{Introduction}
\subsection{Motivation and the Formulation of the Problem}
In the presence of a large number of players, equilibrium computation for Markov games is known to be intractable in general. Mean-field games (MFGs) provide a somewhat tractable framework for approximating the equilibria of Markov games when the players are \emph{homogeneous} and their interaction is \emph{weak}. Even in the discrete-time setting, the forward--backward coupling inherent in MFGs presents significant computational challenges and fundamental barriers to equilibrium computation \cite{lauriere2021numerical,lauriere2022learning}.

In this paper, we study the approximation of discrete-time infinite-horizon stationary discounted and average-cost mean-field equilibria (MFEs) under regularization via a class of extended regularized finite-horizon MFEs. Even in the regularized setting, which introduces a deviation from the true equilibria, the dynamical systems characterizing stationary equilibria are known to be contractive only when the Lipschitz parameters are sufficiently small \cite{guo2019learning,anahtarci2023q}. To approximate these equilibria {under a less restrictive contractive condition}, we study the limiting dynamical system obtained from the ``tail'' of a finite-horizon mean-field game (MFG) under a prescribed terminal $Q$-function. We refer to this limiting problem as a \emph{mean-field terminal value problem} (MFTVP).

Formally, an (infinite-horizon) MFTVP has a configuration similar to that of a discrete-time MFG \cite{SaBaRaSIAM}. It is specified by a tuple \((X,A,c,p,Q_0)\), where \(X\) is the state space, \(A\) is the action space, \(c:X\times A \times \mathcal P(X) \to \mathbb R\) is a one-stage cost function, and \(p:X\times A \times \mathcal P(X) \to \mathcal P(X)\) is a transition kernel. The additional parameter \(Q_0\) is specific to the MFTVP formulation and represents a prescribed terminal \(Q\)-function. Since our primary interest is computational, we assume throughout this work that \(X\) and \(A\) are finite sets equipped with the discrete metric.

{Unlike a standard discrete-time MFG \cite{SaBaRaSIAM}, an MFTVP is constructed by fixing a terminal \(Q\)-function and, in the infinite-horizon setting, extending the finite-horizon backward \(Q\)-function recursion indefinitely into the past. Thus, rather than evolving forward from a prescribed initial state distribution, the \(Q\)-functions evolve backward from the prescribed terminal condition \(Q_0\), with indices \(0,-1,-2,\ldots\). This indefinite backward recursion does not, in general, correspond to} the dynamic programming equations of a single global objective function for the representative agent. We therefore define an MFTVP directly through the coupled evolution of the \(Q\)-functions and population state measures \((Q_{-t},\mu_{-t})_{t\in\mathbb N}\). In particular, given a flow of state measures \((\mu_{-t})_{t=0}^{\infty}\), the \(Q\)-functions evolve backward according to backward dynamics inspired by the dynamic programming characterization of the objective function in discrete-time Markov decision processes. In particular, given a flow of state measures (representing the distribution of the population) \((\mu_{-t})_{t=0}^{\infty}\), the \(Q\)-functions \((Q_{-t})_{t=0}^{\infty}\) satisfy
\begin{equation}\label{eq:lim1}
H_1(Q_{-t},\mu_{-t-1})(x,a)=Q_{-t-1}(x,a)=c(x,a,\mu_{-t-1})+\beta\sum_{y \in X}\min_{b \in A}Q_{-t}(y,b)\,p(y \mid x,a,\mu_{-t-1}).
\end{equation}
For each \(t\in\mathbb N\), let \(\pi_{-t}:X\to\mathcal P(A)\) denote a \emph{policy} at time \(t\). Given a policy flow \((\pi_{-t})_{t\in\mathbb N}\), we seek a forward evolution of the state measures satisfying the consistency condition required for a discrete-time MFE:
\begin{equation}\label{eq:lim2}
H_2(\pi_{-t-1},\mu_{-t-1})=\mu_{-t}(\cdot)=\sum_{x \in X}\sum_{a \in A} p(\cdot \mid x,a,\mu_{-t-1})\,\pi_{-t-1}(a\mid x)\,\mu_{-t-1}(x).
\end{equation}
Unlike in the standard MFG setting, we index time by the negative integers to emphasize the backward evolution of the system. The \(Q\)-function \(Q_0\) is referred to as the terminal \(Q\)-function of the MFTVP and is a part of the problem configuration. A solution to an MFTVP is given by the following equilibrium notion.

\begin{definition}
Let \((X,A,c,p,Q_0)\) be an MFTVP. A \emph{time-reversed mean-field equilibrium} (TRMFE) obtained under a fixed terminal \(Q\)-function \(Q_{\mathrm{terminal}}\) is a family \((\pi_{-t},\mu_{-t})_{t=0}^{\infty}\) such that there exists a family of \(Q\)-functions \((Q_{-t})_{t=0}^{\infty}\), with \(Q_0 = Q_{\mathrm{terminal}}\), satisfying
\[
H_1(Q_{-t},\mu_{-t-1}) = Q_{-t-1},
\qquad
H_2(\pi_{-t-1},\mu_{-t-1}) = \mu_{-t},
\]
and \(\pi_{-t-1}(\cdot\mid x){\in \mathcal P(A)}\) is supported on the optimal actions of \(Q_{-t-1}(x,\cdot)\).
\end{definition}

{
\begin{remark}
Compared with the terminal cost problem, in which the terminal cost depends only on the state, our formulation considers a specific terminal cost that also depends on the action.
\end{remark}
}

In what follows, the terminal \(Q\)-function \(Q_0\) will be fixed, say \(Q_0 = Q_{\mathrm{terminal}}\). To emphasize the presence of a fixed terminal \(Q\)-function, we will denote an MFTVP by \((X,A,c,p,Q_{\mathrm{terminal}})\), or by \(\mathrm{MFTVP}_{Q_{\mathrm{terminal}}}\) when the tuple \((X,A,c,p)\) is fixed. Thus, the study of an MFTVP amounts to establishing the existence and uniqueness of a TRMFE under the prescribed terminal \(Q\)-function.

As mentioned, an MFTVP captures the tail behavior of a finite-horizon discrete-time MFG. In the finite-horizon setting, given an initial state measure \(\mu_{-T}\), a (forward) discrete-time MFG is described by a family \((\pi^T_{-t},\mu^T_{-t})_{t=0}^{T}\), where, for all \(t\le T\), we have
\[
\mu^T_{-t}(\cdot) = \sum_{x\in X}\sum_{a \in A}p(\cdot \mid x,a,\mu^T_{-t-1})\,\pi^T_{-t-1}(a\mid x)\,\mu^T_{-t-1}(x),
\]
and \(\pi_{-t}\) is concentrated on the optimal state-action pairs of the \(Q\)-function
\[
Q^T_{-t-1}(x,a) = H_1(Q^T_{-t},\mu^T_{-t-1})
= c(x,a,\mu^T_{-t-1}) + \beta\sum_{y\in X}\min_{b \in A}Q^T_{-t}(y,b)\,p(y \mid x,a,\mu^T_{-t-1}),
\]
for \(t\le T-1\). In the finite-horizon setting, we refer to \(Q^T_0\) as the terminal \(Q\)-function. For a (forward) MFG, one takes \(Q^T_0 \equiv 0\), which defines the objective function of the representative agent as
\[
J((\pi^T_{-t})_{t=0}^T)
=
E\left[
\sum_{t=0}^{T}
\beta^t c(x_{-T+t},a_{-T+t},\mu_{-T+t})
\right].
\]
\begin{definition}\label{def:fh}
    {A finite-horizon MFTVP is specified by a tuple \((X,A,c,p,Q_0,\mu_{-T})\). A solution consists of \(Q\)-functions \((\tilde Q^T_{-t})_{t=0}^{T}\), state measures \((\tilde\mu^T_{-t})_{t=0}^{T}\), and optimal policies \((\tilde\pi^T_{-t})_{t=0}^{T}\) satisfying
\[
\tilde Q^T_{-t-1}
=
H_1(\tilde Q^T_{-t},\tilde\mu^T_{-t-1}),
\qquad
\tilde Q^T_0=Q_0,
\]and
\[
\operatorname{supp}\tilde\pi^T_{-t}(\cdot\mid x)
\subseteq
\operatorname*{arg\,min}_{a\in A}
\tilde Q^T_{-t}(x,a).
\]The state measures evolve according to
$
\tilde\mu^T_{-t}
=
H_2(\tilde\pi^T_{-t-1},\tilde\mu^T_{-t-1}),
$
where
$
\tilde\mu^T_{-T}=\mu_{-T}.
$}
\end{definition}

\begin{remark}
    {Unlike in the infinite-horizon MFTVP, the initial state measure is fixed in a finite-horizon MFTVP.}
\end{remark}

{Even in the finite-horizon setting, a nontrivial terminal \(Q\)-function in an MFTVP need not correspond to a terminal cost because it is prescribed externally rather than derived from the representative agent’s objective function. Consequently, the resulting backward \(Q\)-recursion need not correspond to a single global objective function.

The infinite-horizon MFTVP can instead be viewed as a limit of these finite-horizon problems near their terminal times. More precisely, suppose that there exists a sequence of horizon lengths \(T_n\to\infty\) such that, for every fixed \(k\in\mathbb N\), the sequences \(\bigl(Q^{T_n}_{-k}\bigr)_n\) and \(\bigl(\mu^{T_n}_{-k}\bigr)_n\) converge. Passing to the limit for each fixed \(k\) yields an infinite family \(\bigl(Q_{-k},\mu_{-k}\bigr)_{k\in\mathbb N}\) satisfying the same coupled backward–forward relations and hence defining an infinite-horizon MFTVP.}

The primary goal of studying MFTVP is to develop algorithms for finite-horizon MFGs that converge to a stationary discounted or average-cost MFE. Indeed, it is easy to see that if, for any given terminal \(Q\)-function \(Q_0 = Q_{\mathrm{terminal}}\), there exists a unique TRMFE and \(Q_0 = Q_{-1}\), then one readily obtains a stationary MFE from the corresponding TRMFE. If the finite-horizon MFEs are constructed with terminal \(Q\)-functions satisfying \(Q_T = H_1(Q_T,\mu_T)\), then the finite-horizon MFEs on the horizons \(\{0,\ldots,T-1\}\) and \(\{0,\ldots,T\}\) start from the same terminal \(Q\)-function. Under mild assumptions on the system components \((c,p)\), this property can then be used to show that they converge to a stationary MFE. Taking this convergence for granted, the preceding program raises the following questions:
\begin{center}
\begin{enumerate}
    \item[Q1)] Under what conditions does there exist a unique TRMFE for any given terminal \(Q\)-function?
    \item[Q2)] Under what conditions can we find a finite-horizon MFE \((\pi_{-t},\mu_{-t})_{t=0}^T\) such that the terminal \(Q\)-function satisfies \(Q_0^T = H_1(Q_0^T,\mu_0^T)\)?
\end{enumerate}

\end{center}

To answer these questions, we construct a contractive mapping in the finite-horizon setting to search for a terminal \(Q\)-function satisfying
\(
Q_0 = H_1(Q_0,\mu_0)
\)
under a finite-horizon MFE. Under suitable Lipschitz continuity assumptions, we show that this mapping is contractive for any initial state measure whenever the finite-horizon MFE is contractive and the discount factor or the drift factor is sufficiently small. {In particular, provided that a unique TRMFE exists for every prescribed terminal \(Q\)-function, the finite-horizon contraction conditions, which are less restrictive than those required for stationary MFEs \cite{ayd}, yield an approximation method for stationary MFEs under weaker conditions.}

In contrast, when this smallness condition is violated, but the finite-horizon MFE remains contractive for every fixed terminal \(Q\)-function and every horizon length, we show that there exists a unique TRMFE associated with any prescribed terminal \(Q\)-function. The requirement of a sufficiently large discount factor is natural, as a strong coupling between successive \(Q\)-functions is needed to ensure that the backward evolution is uniquely determined by a given terminal \(Q\)-function.

On the other hand, when the discount factor is sufficiently small, a straightforward perturbation argument yields the desired contraction while preserving the finite-horizon MFE contraction property. Although the uniqueness of TRMFE requires a sufficiently large discount factor, constructing a terminal \(Q\)-function satisfying
\(
Q_T = H_1(Q_T,\mu_T)
\)
for arbitrarily large horizon lengths requires working in the complementary regime of sufficiently large discount factors.

Unlike Markov decision processes, discrete-time MFGs do not admit a natural contraction condition that is always satisfied by iterative algorithms. Consequently, it is desirable to develop algorithms that rely on weaker and less restrictive contraction conditions. In some settings, the contraction conditions required for the convergence of our proposed procedures are weaker than those needed to compute a stationary MFE directly, which is the main motivation for our study. From an application standpoint, our algorithms require knowledge of the underlying model and are therefore applicable to model-based learning schemes.
\subsection{Literature Review}
\paragraph{Mean-field Games} Mean-field games (MFGs) provide a powerful framework for approximating solutions to multi-agent problems with identical agents. The theory of MFGs was introduced independently by Lasry and Lions in \cite{LaLi07}, who coined the term ``mean-field games,'' and by Huang, Malham\'{e}, and Caines in \cite{HuMaCa06}, who formulated the problem from the viewpoint of stochastic dynamic games. These foundational have studied continuous-time noncooperative differential games with a large but finite number of agents, where the influence of each individual agent on the population becomes asymptotically negligible as the number of agents grows. The discrete-time counter part of MFGs that our framework relies on has been introduced in \cite{SaBaRaSIAM}. Due to the similarity of the structure of discrete-time MFGs with Markov decision processes (MDPs), adaptation of reinforcement learning techniques to compute MFE has gained attraction \cite{guo2019learning, subramanian2019reinforcement}. However, unlike MDPs \cite{hernandez2012adaptive}, discrete-time MFGs do not admit a natural global asymptotically stable equilibrium, which most learning algorithms rely on \cite{borkar1997stochastic, guo2019learning, anahtarci2023q,anahtarci2020value,cui2021approximately, subramanian2019reinforcement}. In fact, the work \cite{yardim2024mean} has demonstrated that under mere Lipschitz conditions, calculating stationary MFE belongs to the PPAD-complete complexity class, which is believed to be intractable \cite{papadimitriou1994complexity}.

\paragraph{Regularized MFGs} Due to lack of natural contraction conditions and the presence of computational complexity barriers, studying regularized MFGs has become a natural approach \cite{guo2022entropy,zhang2023learning,xie2021learning}. In our setting, regularization refers to perturbing the cost function with an appropriate regularizer to ensure that the resulting policies possess desirable properties, such as Lipschitz continuity. In the MFG literature, entropy regularization was first introduced in \cite{guo2019learning} through the use of ``soft-min'' policies. The work \cite{cui2021approximately} showed that soft-min policies are equivalent to entropy regularization in the MFG setting.

More generally, \cite{anahtarci2023q} studied a broad class of regularizers and established contraction conditions for regularized stationary MFGs in terms of the Lipschitz coefficients. These conditions require the Lipschitz parameters to be sufficiently small and deteriorate as the discount factor vanishes. Along similar lines, \cite{ayd} showed that, in the finite-horizon setting, one can establish contraction conditions for regularized non-stationary MFGs that do not deteriorate as the discount factor approaches one.

Our work is closely related to \cite{ayd}. We adapt the state-measure iteration proposed therein and extend it to an iteration of the \(Q\)-functions. Building on this analysis, we show that whenever the discount factor or the drift factor is sufficiently small, the invariance condition \(Q_0^T \equiv Q_{-1}^T\) holds under the finite-horizon MFE contraction condition of \cite[Theorem~2]{ayd}, which is less restrictive than the corresponding contraction condition for stationary MFEs established in \cite{anahtarci2023q}. Finally, we note that, beyond contraction-based methods, optimization formulations of MFGs \cite{guo2024mf} have also provided computational tools for computing non-stationary MFEs.

\paragraph{Finite-Horizon Approximations of Infinite-Horizon MFE} Our contraction conditions are established for the finite-horizon setting; therefore, one must justify the convergence of finite-horizon TRMFEs to their infinite-horizon counterpart. Unlike in Markov decision processes, proving the convergence of the value function alone \cite{kara2019robustness} is insufficient to guarantee the convergence of the MFE. For instance, in calculus of variations and game theory, it is well known that relaxed controls do not necessarily converge to the corresponding optimal controls \cite{milgrom1985distributional,braides2002gamma}. In the continuous-time setting, this finite-to-infinite-horizon convergence problem has been studied for the FBSDE formulation of MFGs under monotonicity conditions \cite{carmona2024probabilistic}. In the discrete-time setting, \cite{ayd} provides sufficient conditions, expressed in terms of Lipschitz parameters, for this convergence together with explicit convergence rates.

We emphasize that the {one-step invariance} condition \(Q_0^T \equiv Q_{-1}^T\), which is central to our approach, has no meaningful analogue in the continuous-time setting. Furthermore, monotonicity conditions are typically global in nature, whereas the condition \(Q_0^T \equiv Q_{-1}^T\) is local. Consequently, monotonicity alone cannot enforce the invariance property that we seek.

Another difficulty that arises in the infinite-horizon setting is the computational intractability of non-stationary solutions. The error bounds established in \cite{ayd} suggest that, at time step \(t\), the approximation error between finite-horizon and infinite-horizon non-stationary MFEs is of the form \(O(a^t b^T)\), where \(0 < a < 1\), \(b > 1\), and \(T\) denotes the horizon length. Consequently, even finite-horizon approximations cannot be used to efficiently compute infinite-horizon non-stationary MFEs. In contrast, stationary solutions possess a much simpler structure \cite{bacsar2026mean} and can themselves be used to approximate sufficiently long finite-horizon non-stationary MFEs \cite{bacsar2026mean}.

\subsection{Contributions}
Motivated by the structural restrictions and computational complexity barriers associated with computing stationary MFEs, we propose a method for approximating stationary MFEs via MFTVPs. The convergence of finite-horizon MFTVPs to their infinite-horizon counterpart, together with uniqueness of the latter, constitutes the primary qualitative contribution of this work. The computation of a TRMFE satisfying \(Q_0 = Q_{-1}\) is achieved by establishing a contraction condition, which constitutes our primary quantitative contribution.

Our contributions can be summarized as follows:
\begin{enumerate}
    \item It appears that, MFTVPs have not been studied previously in the literature in the infinite-horizon setting. Although finite-horizon MFGs with terminal conditions can be viewed as a special case of our problem, even in the finite-horizon setting, this problem has not been studied at the level of generality considered here. We provide existence (Theorem~\ref{thrm:a}) and uniqueness (Theorem~\ref{thrm:b}) conditions for time-reversed MFEs in the average-cost and discounted stationary settings. Since our primary motivation is to establish an approximation scheme between finite-horizon and stationary {discounted cost}/average-cost MFEs, we impose relatively restrictive assumptions to keep the exposition simple.
    \item To provide a computationally tractable method for computing a TRMFE under a terminal \(Q\)-function \(Q_0\) satisfying \(Q_0 = Q_{-1}\), we show that the infinite-horizon MFTVP can be approximated by finite-horizon MFTVPs. The resulting stability result (Theorem~\ref{thrm:1}) is then used to establish the existence of a TRMFE associated with a given terminal \(Q\)-function. We emphasize that this convergence is by no means necessary for proving existence; see, for example, \cite{SaBaRaSIAM}. While doing so, we use a trick due to Dobrushin \cite{dobrushin} to obtain improved contraction conditions that are also applicable to previous works in the literature \cite{yardim2023policy,anahtarci2020value, cui2021approximately, lauriere2021numerical}.
    \item We develop a novel approximation scheme for average-cost and {discounted cost} MFEs via finite-horizon {TR}MFEs. To this end, for each horizon length \(T\), we seek a one-step invariant terminal \(Q\)-function (i.e., \(Q_0^T = Q_{-1}^T\)) in the finite-horizon problem. This is achieved by constructing a contractive mapping that simultaneously searches for a terminal \(Q\)-function satisfying this invariance property (Theorem~\ref{thrm:3}).
Whenever there exists a unique TRMFE associated with a terminal \(Q\)-function, we show that the accumulation points of TRMFEs corresponding to terminal \(Q\)-functions satisfying \(Q_0^T = Q_{-1}^T\) are stationary MFEs. Existing approximation schemes in the literature based on (forward) MFGs require a sufficiently large discount factor or drift factor \cite{anahtarci2023q}. In contrast, our approximation scheme requires a small drift factor in the average-cost setting and a small discount factor in the discounted setting, while remaining applicable even in the presence of multiple MFEs. In either case, the required contraction conditions are weaker than that required for stationary MFGs. We provide an explicit example in Section \ref{sec:numerical-example2} to demonstrate that our algorithm converges to a stationary MFE but the previous algorithms in the literature fail to do so \cite{guo2019learning,anahtarci2020value,lauriere2022learning,lauriere2021numerical}.
\end{enumerate}

\section{Discrete-Time Mean-field Games}

Our primary motivation for introducing MFTVP is to develop approximation schemes between finite-horizon MFE and infinite-horizon stationary MFE. To this end, in this section, we review the theory of discrete-time MFGs introduced in \cite{SaBaRaSIAM}, which we will refer to as \emph{forward MFGs} whenever there is potential for confusion with MFTVPs. A forward MFG is characterized by the tuple \((X,A,c,p)\), where \(X\) is a finite state space, \(A\) is a finite action space, \(c:X \times A \times \mathcal P(X) \to \mathbb R\) is the one-stage cost function, and \(p:X \times A \times \mathcal P(X) \to \mathcal P(X)\) is the one-stage transition probability. Throughout this paper, for simplicity, we assume that \(X\) and \(A\) are finite sets. Given a fixed \(\mu_0 \in \mathcal P(X)\), representing the state distribution of the population at time \(t=0\), we denote by
\(
\mathrm{MFG}_{\mu_0} = (X,A,c,p,\mu_0)
\)
a forward discrete-time MFG.

\subsection{Discounted Discrete-Time Mean-field Games}

Let \(1>\beta>0\) be a given discount factor. As mentioned before, a discrete-time MFG is defined by the tuple $\mathrm{MFG}_{\mu_0}=(X,A,c,p,\mu_0)$. Informally, an MFG represents the limiting dynamics obtained from a homogeneous \(N\)-player game with weak interactions from the perspective of a representative player \cite{LaLi07,HuMaCa06}. The given input \(\mu_0\) represents the initial state distribution of a continuum of players obtained from the limiting game obtained from some \(N\)-player game. 

A solution to an infinite-horizon $\mathrm{MFG}_{\mu_0}$ is a pair consisting of a state-measure flow $\pmb \mu=(\mu_0,\mu_1,\cdots) \in \prod_{t=0}^{\infty}\mathcal P(X)$ and a policy flow $\pmb \pi = (\pi_0,\pi_1,\pi_2,\cdots)\in \prod_{t=0}^{\infty}\mathcal P(A)^X$ such that
\begin{equation}\label{eq:const1}
\mu_{t+1}(\cdot)=H_2(\pi_t,\mu_t)(\cdot) := \sum_{x \in X} \sum_{a \in A} p(\cdot \mid x,a,\mu_t)\,\pi_t(da\mid x)\,\mu_t(dx),
\end{equation}
and the corresponding objective function $J$ of the representative agent given by
\begin{equation}\label{eq:const3}
J(\tilde{\pmb \pi},\pmb \mu) = E_{\tilde{\pmb \pi}}\!\left[ \sum_{t=0}^\infty \beta^t c(x_t,a_t,\mu_t) \right],
\end{equation}
satisfying the inequality
\begin{equation}\label{eq:const2}
J(\tilde{\pmb \pi},\pmb \mu) \ge J(\pmb \pi,\pmb \mu)
\end{equation}
for any admissible policy $\tilde{\pmb \pi}$. In \eqref{eq:const3}, the evolutions of the states and actions are given by
\[
x_0 \sim \mu_0, \, a_t \sim \widetilde \pi_t(\cdot \mid x_t), \, x_{t+1} \sim p(\cdot \mid x_t,a_t,\mu_t).
\]
Informally, an MFE optimizes the objective of a representative player subject to a state-measure evolution that must be satisfied by the population distribution, for which the weak interaction assumption is crucial \cite{SaBaRaSIAM}. If the pair $(\pmb \pi, \pmb \mu)$ satisfies the constraints \eqref{eq:const1} and \eqref{eq:const2}, we call it a \emph{(discounted) mean-field equilibrium} (MFE). We denote the set of mean-field equilibria associated with $\mathrm{MFG}_{\mu_0}$ by $\mathrm{MFE}_{\mu_0}$ in the infinite-horizon non-stationary setting. Due to the required interaction between state measures and optimal policies, an MFE is neither a game equilibrium in the traditional sense nor the solution to a standard stochastic optimal control problem.

Given a flow of state-measures \((\mu_t)_t \in \prod_{t=0}^{\infty} \mathcal P(X)\), similarly to Markov decision processes \cite{hernandez2012discrete}, dynamic programming \cite{SaBaRaSIAM} implies that, under mild assumptions, there exists a unique family of $Q$-functions $(Q_t)_{t=0}^{\infty}$ such that
\begin{equation}\label{eq:discc}
Q_t(x,a)=c(x,a,\mu_t) + \beta \sum_{y \in X}\min_{b \in A}Q_{t+1}(y,b)\,p(y \mid x,a,\mu_t),
\end{equation}
where the expected value of $Q_0$ under the policy flow $\pmb \pi=(\pi_0,\pi_1,\cdots)$ recovers the value of the objective function $J(\pmb \pi,\pmb \mu)$ with respect to the given initial state-measure.

In the finite-horizon case, when the horizon length is $T$, one is instead interested in the objective function
\begin{equation}\label{eq:const4}
J(\tilde{\pmb \pi},\pmb \mu) = E_{\tilde{\pmb \pi}}\!\left[ \sum_{t=0}^{T} \beta^t c(x_{t},a_{t},\mu_{t}) \right].
\end{equation}
In the reinforcement learning literature \cite{cui2021approximately}, one often takes \(\beta=1\) in the finite-horizon setting. However, since we are interested in approximating infinite-horizon MFGs, our finite-horizon MFGs also depend on the same discount factor. In this case, under the dynamic programming formulation, the last $Q$-function in the recursion, $Q_{T}$, is given by
\(
Q_{T}(x,a)=c(x,a,\mu_{T}),
\)
i.e., there is no terminal cost. Usually, in the MDP setting, one is interested in terminal costs that depend only on the state. In our case, instead, we introduce a ``terminal $Q$-function'' by imposing the condition
\begin{equation}\label{eq:const5}
Q_{T}(x,a) = c(x,a,\mu_{T}) + \beta \sum_{y \in X}\min_{b \in A}Q_{\mathrm{terminal}}(y,b)\,p(y \mid x,a,\mu_{T})
\end{equation}
on $Q_{T}$ at the terminal stage. This creates an indefinite loop in the finite-horizon MFG at the terminal stage and therefore violates \eqref{eq:const4}. Consequently, an MFTVP does not necessarily admit a closed-form expression for its objective function. At this stage, the motivation for introducing the terminal \(Q\)-function is to provide a method for approximating \emph{stationary MFEs}; it is introduced purely for technical reasons to provide an additional invariant for computational purposes.

\begin{definition}
    Let \((X,A,c,p,\mu_0)\) be an infinite-horizon \(\beta\)-discounted MFG. We say that an MFE \( (\pmb \pi,\pmb \mu)\) is stationary if \(\mu_t = \mu_0\) and \(\pi_t = \pi_0\) for all \(t \in \mathbb N\).
\end{definition}

In general, when one seeks a stationary MFE, the initial state-measure \(\mu_0\) is part of the solution rather than part of the formulation. Thus, a \emph{stationary MFG} can also be referred to as the tuple \((X,A,c,p)\). For Markov decision processes, stationary solutions involve only the existence of a stationary policy in the infinite-horizon setting and are known to exist under mild assumptions \cite{hernandez2012discrete}. Such solutions are often preferred in practice due to their simpler structure. Since an MFE involves the distribution of a population, relaxing the constraint on \(\mu_0\) often breaks the formulation of the problem. However, it has also been observed in the MFG literature that non-stationary solutions are sometimes quite difficult to compute, whereas stationary solutions are easier to compute. One example is the renewable energy model formulated as a mean-field game in \cite[Subsection 15.7.1]{bacsar2026mean2}. Nevertheless, there are also practical applications of stationary MFEs. For instance, \cite{weintraub2008markov} models an investment problem for firms that make decisions based on the average quality of the other firms. The interaction among the firms is modeled as a finite-agent stochastic game with mean-field interactions. Their notion of an ``oblivious equilibrium'' requires sampling from a distribution that is not given a priori, which fits the definition of a stationary MFE.

The introduction of the terminal $Q$-function breaks the structure of the objective functions \eqref{eq:const3} and \eqref{eq:const4} for all $T \in \mathbb N\cup\{\infty\}$. However, these ``finite-horizon MFGs'' defined under a terminal $Q$-function allow us to exploit the discrete-time structure of MFGs to establish an approximation scheme between finite-horizon MFE and stationary and average-cost MFE by treating the terminal \(Q\)-function as an invariant. To see this, first note that when the initial state-measure $\mu_0 \in \mathcal P(X)$ is fixed for a finite-horizon MFE (or, for that matter, an infinite-horizon non-stationary MFE), $\mu_0$ can be viewed as an invariant of the corresponding dynamical system
\[
\begin{bmatrix}
    Q_{t+1} \\
    \mu_{t+1}
\end{bmatrix}
=
\begin{bmatrix}
    H_1(Q_{t+2},\mu_{t+1}) \\
    H_2(Q_{t},\mu_t)
\end{bmatrix},
\qquad \forall t \ge 0.
\]
By introducing the terminal $Q$-function in \eqref{eq:const5}, we impose a new invariant on the system by means of the $Q$-functions in the finite-horizon setting.

A state-action pair $(x,a) \in X\times A$ is said to be \emph{optimal} for a function $Q:X\times A \to \mathbb R$ if
\[
a \in \arg\min_{b \in A} Q(x,b)=: \mathrm{Opt}(Q(x,\cdot)).
\]
We say that a policy $\pi:X \to \mathcal P(A)$ \emph{concentrates} on the optimal state-action pairs of a $Q$-function $Q:X\times A \to \mathbb R$ if $\pi(\cdot\mid x)$ is supported on $\mathrm{Opt}(Q(x,\cdot))$ for all $x \in X$. We say that a flow of state-measures $(\mu_t)_{t \in \mathbb N} \in \mathcal P(X)^{\mathbb N}$ \emph{induces} a $\mathrm{MFE}_{\mu_0}$ if there exists a family of policies $(\pi_t)_{t \in \mathbb N}$ such that $(\pi_t,\mu_t)_{t \in \mathbb N} \in \mathrm{MFE}_{\mu_0}$. Similarly, we say that $(\mu^T_t)_{t=0}^{T-1} \in \mathcal P(X)^{T-1}$ \emph{induces} a $\mathrm{MFE}_{T,\mu_0,Q_{\mathrm{terminal}}}$ if there exists a family of policies $(\pi^T_t)_{t=0}^{T-1}$ and a corresponding family of $Q$-functions $(Q_t)_{t=0}^{T-1}$ satisfying $Q_t=H_1(Q_{t+1},\mu^T_t)$ and $Q_{T-1}=H_1(Q_{\mathrm{terminal}},\mu^T_{T-1})$ such that, for all $t=0,\ldots,T-1$, the policy $\pi^T_t$ concentrates on the optimal state-action pairs of $Q_t$.

For a given terminal $Q$-function $Q_{\mathrm{terminal}}$, we denote by $\mathrm{MFG}_{T,\mu_0,Q_{\mathrm{terminal}}}$ the finite-horizon MFG that starts from the initial state-measure $\mu_0$ and terminates with the terminal $Q$-function $Q_{\mathrm{terminal}}$ after $T$ steps. An MFE of $\mathrm{MFG}_{T,\mu_0,Q_{\mathrm{terminal}}}$ is a flow of policies and state-measures $(\pi_t,\mu_t)_{t=0}^T$ such that there exists a family of $Q$-functions $(Q_t)_{t=0}^{T-1}$ satisfying $Q_t=H_1(Q_{t+1},\mu_t)$ and $Q_{T-1}=H_1(Q_{\mathrm{terminal}},\mu_{T-1})$, where $\pi_t$ is concentrated on the optimal state-action pairs of $Q_t$. By $\mathrm{MFE}_{T,\mu_0,Q_{\mathrm{terminal}}}$, we denote the set of MFEs of $\mathrm{MFG}_{T,\mu_0,Q_{\mathrm{terminal}}}$, which is formally defined as
\begin{align*}
\mathrm{MFE}_{T,\mu_0,Q_{\mathrm{terminal}}}
&= \bigg\{ 
(\pi_t,\mu_t)_{t=0}^T : \exists (Q_t)_{t=0}^{T-1} \text{ s.t. } Q_t =H_1(Q_{t+1},\mu_t), \, Q_{T-1}=H_1(Q_{\mathrm{terminal}},\mu_{T-1}),
\\& \qquad \qquad \qquad \qquad \qquad\mu_{t+1}=H_2(Q_{t},\mu_t), \, \mathrm{supp}\,\pi_t(\cdot \mid x)\subset \mathrm{Opt}(Q_t(x,\cdot)), \, \forall x\in X \bigg\}.
\end{align*}

\begin{convention}
We will reindex the vectors $(\pi_t,\mu_t)_{t=0}^{T-1} \in \mathrm{MFE}_{T,\mu_0,Q_{\mathrm{terminal}}}$ as $(\pi_{-t},\mu_{-t})_{t=0}^{T-1}$ by shifting the indices by $T-1$. This fixes the terminal time index at $0$, which allows us to define a finite terminal time for $\mathrm{TRMFE}_{Q_{\mathrm{terminal}}}$ while approximating it via $\mathrm{MFE}_{T,\mu_0,Q_{\mathrm{terminal}}}$.
\end{convention}

For our algorithmic purposes, we require our policies to be \emph{deterministic}. This can be achieved by transforming the MFG into an equivalent one following \cite{anahtarci2023q}. We say that a policy $\pi$ is deterministic if there exists a function $f:X \to A$ such that $\pi(\cdot|x) = \delta_{f(x)}(\cdot)$, where $\delta_{f(x)}(\{a\}) := 1_{\{f(x)\in \{a\}\}}$. To ensure that our $Q$-functions always have a unique minimizer, we lift the action space from $A$ to $\mathcal P(A)$ by defining $C:X\times \mathcal P(A)\times \mathcal P(X) \to \mathbb R$ and $P: X \times \mathcal P(A) \times \mathcal P(X) \to \mathcal P(X)$, where for \(u \in \mathcal P(A)\),
\[
C(x,u,\mu) := \sum_{a \in A}c(x,a,\mu)u(a) \text{ and } P(\cdot \mid x,u,\mu) := \sum_{a \in A} p(\cdot \mid x,a,\mu)u(a).
\]
In the setting $(X,\mathcal P(A),C,P)$, the policies obtained from $(X,A,c,p)$ correspond to deterministic policies. Let $\Omega:\mathcal P(A) \to \mathbb R$ be $1$-strongly convex with respect to the $1$-norm, i.e., for all $u,v \in \mathcal P(A)$ and $\theta \in [0,1]$, it holds that
\[
\Omega(\theta x+(1-\theta)y) \leq \theta \Omega(x) +(1-\theta)\Omega(y) - \frac 12 \theta(1-\theta) \| x-y\|^2_1.
\]
Then, for a given $\rho>0$, we define a \emph{regularized MFG} as the tuple $(X,\mathcal P(A),C+\rho\Omega,P)$. Although the MFEs obtained from the system $(X,\mathcal P(A),C+\rho\Omega,P)$ deviate from those of $(X,\mathcal P(A),C,P)$, the corresponding $Q$-functions have a unique policy that is Lipschitz continuous in the setting $(X,\mathcal P(A),C+\rho\Omega,P)$.

\begin{convention}
From now on, whenever we refer to the operators \(H_1\) and \(H_2\), we take the cost function to be \(C+\rho\Omega\) and the transition probability kernel to be \(P\).
\end{convention}

\subsection{Average-Cost Discrete-time Mean-field Games}\label{sec:avg}
    In the average-cost case (sometimes also referred to as the ergodic case), the objective function of the representative agent is given as
\begin{equation}\label{eq:const6}
J(\tilde{ \pi}, \mu) = \limsup_{T \to \infty} \frac 1TE_{\tilde{ \pi}}\!\left[ \sum_{t=0}^{T-1} c(x_t,a_t,\mu) \right]
\end{equation}
instead of the discounted objective function \eqref{eq:const4}. However, the required consistency condition for \(\mu\) stays the same:
\[
\mu ( \cdot ) = \sum_{ x\in X} \sum_{ a\in A} p(\cdot \mid x,a,\mu)\,{\tilde \pi}(a\mid x)\,\mu(x).
\]
Since the objective function \eqref{eq:const6} does not involve a discount factor, the Bellman equation \eqref{eq:const5} cannot be used to characterize the corresponding $Q$-functions. However, under additional assumptions on the dynamics, one can characterize the $Q$-functions through the so-called \emph{asymptotic discount optimality equation} \cite{hernandez2012further}.

\begin{assumption}\label{ass:acmfg}
\begin{enumerate}
\item []
\item There exists a sub-probability measure $\lambda$ on $X$ such that
\(
\inf_{x,a,\mu}p(\,\cdot \mid x,a,\mu) \ge \lambda(\,\cdot\,).
\)

\item For some $\omega :X \times A \to [1,\infty)$, there exist a non-negative constant $\alpha \in (0,1)$ (called the drift factor)
and a non-negative real number $b$ such that
\[
\sum_{y \in X} \omega _{\max}(y)\, p(y \mid  x,a,\mu) \le \alpha\, \omega(x,a) + b,
\]
where \(\omega_{\max}(x):=\max_{a\in A}\omega(x,a)\).
\end{enumerate}
\end{assumption}

The conditions in Assumption \ref{ass:acmfg} ensure that the dynamics return to the ``center'' of the state space regularly, uniformly over all population distributions, while the durations of excursions away from this region remain tightly controlled. Moreover, there exists a sufficiently large level set of \(\omega\) that satisfies a minorization condition \cite{hairer2011yet}. For our purposes, we are interested in the regime where \(\alpha\) is sufficiently small, meaning that the state dynamics are strongly pulled back toward smaller values of the weight function \(\omega\), uniformly over all population distributions. The required threshold for \(\alpha\) will be made explicit in the next section.

{
\begin{definition}\label{lem:-1}
By \(\tilde p :X \times A \times \mathcal P(X) \to \mathcal P(X)\), we denote the sub-probability kernel
\(
\tilde p(\cdot \mid x,a,\mu) = p(\cdot \mid x,a,\mu) - \lambda (\cdot).
\)
\end{definition}
}

Similar to the discounted case, we study the regularized versions of these average-cost MFGs. In particular, we define
\(
\tilde P(\cdot \mid x,u,\mu) = \sum_{a \in A} \tilde p(\cdot \mid x,a,\mu)u(a).
\)
It is straightforward to verify that \( P\) satisfies Assumption \ref{ass:acmfg} with the weight function \(\tilde \omega (x,u) = \sum_{a \in A} \omega(x,a)u(a)\) for \(u \in \mathcal P(A)\). By a slight abuse of notation, since we always work under regularization, we write \(\omega\) instead of \(\tilde \omega\) for the lifted weight function.

\begin{convention}\label{conv}
By adjusting \(\alpha\), we can always ensure that
\(
\sum_{x\in X}\sup_{a \in A}\omega(x,a)\lambda(x) = b.
\)
We assume this implicitly throughout our average-cost results. See \cite[Remark~1]{anahtarci2020value}.
\end{convention}

Under Assumption \ref{ass:acmfg}, the objective function in \eqref{eq:const6} can be characterized via the $Q$-function
\begin{equation}\label{eq:const8}
Q(x,a) = c(x,a,\mu) + \sum_{y \in X}\min_{b \in A}Q(y,b)\,\tilde p(y \mid x,a,\mu).
\end{equation}
Equivalent formulations of \eqref{eq:const8} are referred to as the \emph{{average cost} asymptotic discount optimality equations} in the Markov decision process literature. We note that, under the minorization condition in Assumption \ref{ass:acmfg}, for any given \(\mu \in \mathcal P(X)\), there exists a unique $Q$-function satisfying \eqref{eq:const8} \cite[Section 2.3]{hernandez2012adaptive}.

In the regularized setting, the discount-free variant of MFTVP in the non-stationary case concerns the $Q$-functions of the form
\[
H^{\mathrm{avg}}_1(Q_{-t+1},\mu_{-t})(x,u)=Q_{-t} (x,u) = C(x,u,\mu_{-t})+\rho\Omega(u) + \sum_{y \in X}\min_{v \in \mathcal P(A)}Q_{-t+1}(y,v)\,\tilde P(y \mid x,u,\mu_{-t}),
\]
where we seek state measures that satisfy the evolution
\begin{equation}\label{eq:const7}
H_2(\pi_{-t-1},\mu_{-t-1})(\cdot)=\mu_{-t}(\cdot) = \sum_{x \in X}\int_{u \in \mathcal P(A)} P(\cdot \mid x,u,\mu_{-t-1})\,\pi_{-t-1}(u\mid x)\,\mu_{-t-1}(x), 
\end{equation}
where $\pi_{-t}(\cdot\mid x) \in \mathcal P(\mathcal P(A))$ is supported on $\mathrm{Opt}(Q_{-t}(x,\cdot))$ for all $x \in X$. A pair \((\pi_{-t},\mu_{-t})_t\) satisfying these properties is called a mean-field equilibrium (MFE) in the average-cost setting. We note that \eqref{eq:const7} does not involve the sub-probability kernel $\tilde P$, even though \eqref{eq:const8} does.
\section{Main Results}

Let $(X,\mathcal P(A),C+\rho\Omega,P)$ be a $\rho$-regularized MFG, where $C$ is a nonnegative function (and hence so is $c$). In this section, we state our main results regarding approximating a stationary MFE for this forward MFG, which constitute the primary contributions discussed in the introduction. Most of the proofs are deferred to the appendix.

Since we have changed the system components from $(c,p)$ to $(C+\rho\Omega,P)$, there is a unique minimizer for each $Q$-function in our setting, at the cost of a deviation from the true MFE. Thus, in practice, taking \(\rho>0\) sufficiently small is more desirable. Similar to \eqref{eq:lim1} and \eqref{eq:lim2}, in the regularized setting we define
\begin{equation}\label{eq:q-it}
H_1(Q,\mu)(x,u):=C(x,u,\mu)+\rho\Omega(u)+\beta\sum_{y \in X}\min_{b \in A}Q(y,b)\,P(y \mid x,u,\mu)
\end{equation}
and
\begin{equation}\label{eq:lim4}
H_2(Q,\mu)(\cdot)=\sum_{x \in X} P(\cdot \mid x,\pi_Q(x),\mu)\,\mu(x),
\end{equation}
where $\pi_Q(x) := \delta_{\mathrm{argmin}_{u \in \mathcal P(A)}Q(x,u)}(\cdot)$. Note that, for any given $Q$-function, the policy $\pi_{H_1(Q,\mu)}$ is well defined and deterministic due to the presence of the regularizer as long as the values of \eqref{eq:q-it} are finite. With this in mind, in the regularized setting the formulations \eqref{eq:lim1}, \eqref{eq:const1}, and \eqref{eq:lim4} are consistent with one another, as we can associate a unique policy with each $Q$-function. However, to simplify the notation, throughout the remainder of this work we use only \eqref{eq:lim4} to describe the evolution of the state measures.

Before defining an appropriate domain for our $Q$-functions so that the operators $H_1$ and $H_2$ are well defined, we impose the following Lipschitz continuity assumptions on the original system components $(c,p)$:

\begin{assumption}\label{ass:1}
    \begin{enumerate}
    \item []
    \item The one-stage cost function \( c:X\times A\times \mathcal P(X) \to \mathbb R_{\ge 0} \) satisfies the following Lipschitz bound:
    \[
    \left| c(x,a,\mu) - c(\hat{x},\hat{a},\hat{\mu}) \right|
    \leq L_1 \left( {1}_{\{x \neq \hat{x}\}} + 2 \,\cdot {1}_{\{a \neq \hat{a}\}} + \|\mu - \hat{\mu}\|_{1} \right),
    \]
    for all \( x, \hat{x} \in X \), all \( a, \hat{a} \in A \), and all \( \mu, \hat{\mu} \in \mathcal P(X) \).

    \item The stochastic kernel \( p :X\times A \times \mathcal P(X) \to \mathcal P(X) \) satisfies the following Lipschitz bound:
    \[
    \left\| p(\cdot | x,a,\mu) - p(\cdot | \hat{x},\hat{a},\hat{\mu}) \right\|_{1}
    \leq K_1 \left( 1_{\{x \neq \hat{x}\}} + 2 \,\cdot 1_{\{a \neq \hat{a}\}} + \|\mu - \hat{\mu}\|_{1} \right),
    \]
    for all \( x, \hat{x} \in X \), all \( a, \hat{a} \in A \), and all \( \mu, \hat{\mu} \in \mathcal P(X) \).
\end{enumerate}
\end{assumption}

Although the assumptions above are stated for the original system components $(c,p)$, they extend to the system components $(C,P)$ in a straightforward manner; see \cite[Proposition 1]{anahtarci2023q}. Let \(M:= \sup_{(x,u,\mu) \in X\times A \times \mathcal P(X)} C(x,u,\mu)\) and \(K:=\sup_{u \in \mathcal P(A)}\Omega(u).\) Note that Assumption \ref{ass:1} implies $M<\infty$. We first construct a candidate domain of $Q$-functions for the original system $(X,A,c,p)$ in the discounted setting (with the discount factor $\beta<1$) as
\begin{align*}
\tilde{\mathcal C}_{\beta} := \Big\{
    Q^{\mathrm{disc}}:X\times A \to \mathbb R_{+}:\;
    &|Q^{\mathrm{disc}}(x,a)-Q^{\mathrm{disc}}(\tilde{x},\tilde{a})|
    \le \frac{L_1}{1 - \frac{\beta K_1}{2}}(1_{\{x \not=\tilde x\}}+2\, \cdot1_{\{a \not = \tilde a\}}),
    \\&
    \qquad\|Q^{\mathrm{disc}}\|_\infty \le \frac{M}{1-\beta}
\Big\},
\end{align*}
where $\| \cdot \|_{\infty}$ denotes the uniform norm on $X \times A$. We then lift this domain to $\mathcal P(A)$ and incorporate the regularizer $\Omega$ as follows:
\begin{align*}
\mathcal C_{\beta} :=
\Big\{
    Q : X \times \mathcal P(A) \to \mathbb R:\;
    &Q(x,u)= \sum_{a \in A} Q^{\mathrm{disc}}(x,a)u(a) + \rho\Omega(u),\;
    Q^{\mathrm{disc}} \in \tilde{\mathcal C}_{\beta}
\Big\}.
\end{align*}

\begin{convention}
   To use this construction in the average-cost case, we define a weighted norm on candidate \(Q\)-functions using the weight function introduced in Assumption \ref{ass:acmfg} as follows:
\[
\| \mathfrak q\|_{\omega} = \sup_{(x,a)\in X\times A} \frac{|\mathfrak q(x,a)|}{\omega(x,a)},
\qquad
\|Q\|_{\widetilde \omega}=\sup_{(x,u)\in X\times\mathcal P(A)}
\frac{|Q(x,u)|}{\widetilde \omega(x,u)},
\]
Then, we say that \(\mathfrak q \in \tilde{\mathcal C}_{\alpha}\) if \(\| \mathfrak q \|_{\omega} \le \frac{M}{1-\alpha }\) and
\(
| \mathfrak q(x,a)-\mathfrak q(\tilde x,\tilde a) |
\leq
\frac{L_1}{1-\frac{K_1\alpha}{2}}
\left(1_{\{x \not = \tilde x\}}+2\, \cdot1_{\{a \not = \tilde a\}} \right).
\)
Similarly, we say that \(Q \in \mathcal C_{\alpha}\) if
\(
Q(x,u)= \sum_{a \in A}\mathfrak q(x,a)u(a)+\rho\Omega(u),
\)
\(
\mathfrak q\in\tilde{\mathcal C}_{\alpha}.
\)
Here, the parameter \(\alpha\) is the drift factor introduced in Assumption \ref{ass:acmfg}.
For the discounted problem under Assumption~\ref{ass:acmfg}, replace $\alpha$ by $\alpha\beta$ in this construction and write $\tilde{\mathcal C}_{\alpha\beta}$ and $\mathcal C_{\alpha\beta}$.
\end{convention}
For any $\mu \in \mathcal P(X)$, the set $\mathcal C_{\beta}$ acts as an invariant domain for $H_1(\cdot,\mu)$. Furthermore, since each $Q$-function in $\mathcal C_{\beta}$ is the sum of a strongly convex function and an affine function over $\mathcal P(A)$, the policy $\pi_Q$ is well defined. In the appendix, we show that $(\mathcal C_{\beta},\| \cdot \|_{\infty})$ is indeed a compact space, a fact that is crucial for establishing the existence of a $\mathrm{TRMFE}_{Q_{\mathrm{terminal}}}$ for every $Q_{\mathrm{terminal}}\in\mathcal C_{\beta}$, as well as for showing that the accumulation points of the tails of $\mathrm{MFE}_{T,\mu_0,Q_{\mathrm{terminal}}}$ are indeed $\mathrm{TRMFE}_{Q_{\mathrm{terminal}}}$.

\begin{remark}
  In the average-cost case (under Assumption \ref{ass:acmfg}), by replacing $\mathcal C_{\beta}$ with $\mathcal C_{\alpha}$, one obtains an invariant domain for the operator
\[
H_1^{\mathrm{avg}}(Q,\mu)(x,u) := C(x,u,\mu)+\rho\Omega(u)+\sum_{y \in X}\min_{v \in \mathcal P(A)}Q(y,v)\,\tilde P(y \mid x,u,\mu).
\]
For the discounted problem under Assumption~\ref{ass:acmfg}, we use the same notation for
\[
H_1^{\mathrm{avg}}(Q,\mu)(x,u) := C(x,u,\mu)+\rho\Omega(u)+\beta\sum_{y \in X}\min_{v \in \mathcal P(A)}Q(y,v)\,\tilde P(y \mid x,u,\mu).
\]
Thus, $\beta=1$ gives the average-cost operator. In particular, in the average-cost case, the terminal $Q$-functions must be chosen from the set $\mathcal C_{\alpha}$. To keep the statements of the main results concise despite the technicalities associated with working with a sub-probability measure, we first present our main results for the discounted-cost case. Unless explicitly stated otherwise, all of these results also hold in the average-cost case.
\end{remark}
\begin{convention}\label{conv:5}
Let \(1>\beta>0\) be a discount factor. Corresponding to \(\beta\), we define the constants
\[
\bar L = \frac{L_1}{1-\frac{\beta K_1}{2}},
\qquad
K_{\mathrm D}:=K_1+\min\left\{1,\frac32K_1\right\}.
\]
Here, $\bar L$ is the state-measure Lipschitz coefficient of $H_1$, while
$K_{\mathrm D}$ is the state-measure Lipschitz coefficient of $H_2$ obtained
by combining the direct mean-field dependence of $P$ with the Dobrushin
coefficient of the induced Markov kernel \cite{dobrushin}. Under
Assumption~\ref{ass:acmfg}, we replace the discount factor \(\beta\) with
\(\alpha\beta\) in the definition of $\bar L$; the definition of
$K_{\mathrm D}$ is unchanged.

{In what follows, \(K_{\mathrm D}\) may be replaced by
\[
\widehat K
:=
\frac{3K_1}{2}
+
\frac{K_1L_1}
{2\rho\left(1-\frac{\beta K_1}{2}\right)};
\]see \cite{ayd,anahtarci2023q}. All results stated in terms of \(K_{\mathrm D}\) remain valid when \(K_{\mathrm D}\) is replaced by \(\widehat K\).}
\end{convention}

    Before stating our first result, we recall that convergence of policies/minimizers is a known obstacle in game theory and variational calculus. In the setting of MFGs, a way to circumvent this difficulty is to work with \emph{occupancy measures} that correspond to optimal policy and state-measure flows \cite{SaBaRaSIAM}.
    Let \(\pi\in\mathcal P(\mathcal P(A))^X\) be a stochastic kernel and \(\mu\in \mathcal P(X)\) be a state measure. These occupancy measures are defined as follows. Given a policy $\pi$ and a state-measure $\mu$, by \(\pi\otimes \mu \in \mathcal P(X\times\mathcal P(A))\) we denote the joint probability measure \(\pi(da|x)\mu(dx)\). We refer to $\pi\otimes \mu$ as an occupancy measure. Although policy flows obtained from approximations of an MFE need not converge to those obtained under an MFE, it is often feasible to show the convergence of the corresponding flow of occupancy measures. By disintegrating the limiting occupancy measure flow, one then obtains an MFE. Using this idea, we obtain our main existence and approximation result for TRMFE. 

\begin{theorem}\label{thrm:a}
     Suppose that Assumption \ref{ass:1} holds. Then, the following results hold.
    \begin{enumerate}
        \item If $Q_{\mathrm{terminal}}\in \mathcal C_{\beta}$, then the set $\mathrm{TRMFE}_{Q_{\mathrm{terminal}}}$ is non-empty.
        \item For all $T \ge0$, let $(\pi^T_{-t},\mu^T_{-t})_{t=0}^{T} \in\mathrm{MFE}_{\mathrm{T},\mu_0,Q^T_{\mathrm{terminal}}}$, where $\mu^T_{-T}=\mu_{\mathrm{initial}} \in \mathcal P(X)$ for all $T\in \mathbb N$ and $Q^T_{\mathrm{terminal}} \in \mathcal C_{\beta}$. Let $\nu \in \mathcal P(X \times \mathcal P(A))$. Extend
        \[
        (\pi^T_{-t}\otimes \mu^T_{-t})_{t=0}^{\infty} =(\cdots,\nu,\nu,\pi^T_{-T}\otimes \mu^T_{-T},\pi^T_{-T+1}\otimes \mu^T_{-T+1},\cdots,\pi^T_0\otimes \mu^T_0) \in \prod_{t=-\infty}^{0}\mathcal P(X\times \mathcal P(A)).
        \]
        If $Q^T_{\mathrm{terminal}} \to Q_{\mathrm{terminal}}$ in $(\mathcal C_{\beta},\|\cdot\|_{\infty})$, then any accumulation point of $\{(\pi^T_{-t}\otimes \mu^T_{-t})_{t=0}^{\infty}\}_{T \ge0}$ in $\prod_{t=0}^{\infty}\mathcal P(X\times \mathcal P(A))$  induces a $\mathrm{TRMFE}_{Q_{\mathrm{terminal}}}$, where each section $\mathcal P(X\times \mathcal P(A))$ is endowed with the topology of the weak convergence.
    \end{enumerate}
\end{theorem}
\begin{proof}
    We defer the proof to Section \ref{sect:4.1}.
\end{proof}

In the result above, the target TRMFE may not be unique; thus, the finite-horizon approximations are only guaranteed along a subsequence.

The main tools behind the proof of Theorem \ref{thrm:a} are the same as those used in \cite{SaBaRaSIAM} to establish the existence of an infinite-horizon non-stationary MFE in a general setting. In particular, we use the ``continuous convergence'' of expressions of the form
\(
\int_{X} f_n(y)\mu_n(dy) \to \int_X f(y) \mu(dy),
\)
when $f_n \to f$ and $\mu_n \to \mu$; see \cite[Theorem 3.5]{serfozo1982convergence} for sufficient conditions ensuring this convergence. We use Assumption \ref{ass:1} to extract convergent subsequences of $Q$-functions, and the continuous convergence criterion allows us to preserve the relation $H_1(Q_{-t+1},\mu_{-t})=Q_{-t}$ for the limiting $Q$-functions. The convergence of the policies required to establish the evolution equation $H_2(Q_{t-1},\mu_{t-1})=\mu_t$ is handled using \cite[Lemma 11]{ayd}; see also \cite[Proposition 3.10]{SaBaRaSIAM}.

We note that Assumption \ref{ass:1} is stronger than the assumptions used in \cite{SaBaRaSIAM} to establish the existence of non-stationary MFE, and Theorem \ref{thrm:a} can be obtained under the assumptions of \cite{SaBaRaSIAM} without altering our proof. However, since our primary motivation is algorithmic, we rely on Assumption \ref{ass:1} to establish our algorithmic convergence results. Therefore, to keep the exposition simple, we state Theorem \ref{thrm:a} under Assumption \ref{ass:1}.

We also emphasize that, in our setting, the proof of Theorem \ref{thrm:a} uses finite-horizon MFE as an approximation to TRMFE. Thus, in addition to establishing existence, the proof of Theorem \ref{thrm:a} also establishes the desired convergence of finite-horizon MFE to TRMFE.

To bypass the need for a subsequence argument when using a TRMFE to approximate a stationary MFE, as is typically required in numerical approximations, a second property needed for TRMFE is uniqueness. Since an infinite-horizon MFTVP depends only on a given terminal $Q$-function, heuristically, a uniqueness result requires a strong connection between the backward iterations of the $Q$-functions. The main parameter governing the interaction between successive $Q$-functions is the discount factor, and our next result shows that there exists a ``large discount factor'' regime in which uniqueness holds, thereby confirming this heuristic.
\begin{theorem}\label{thrm:b}
    Suppose that Assumption \ref{ass:1} holds. Let \(1>\beta>0\) be a discount factor. If it holds that
    \begin{equation}\label{eq:rv1}
    K_{\mathrm D}<1,\,\frac{\bar L K_1}{\rho}<\beta(1-K_{\mathrm D})^2, \text{ and } \sqrt{\beta}\left( \sqrt{K_{\mathrm D}+ \frac{\bar LK_1}{\rho}}+\sqrt{\frac{\bar LK_1}{\rho}} \right)<1,
    \end{equation}
    then under any \(\mathrm{TRMFE}_{Q_{\mathrm{terminal}}}\) for any \(Q_{\mathrm{terminal}}\in\mathcal C_{\beta}\), {the corresponding TRMFE is unique.}
\end{theorem}
\begin{proof}
    We defer the proof to Section \ref{sect:4.3}.
\end{proof}

The proof of Theorem \ref{thrm:b} relies on approximating TRMFE by finite-horizon MFE obtained under the same terminal $Q$-function. For this purpose, we study the differences between finite-horizon MFE and TRMFE using only the corresponding $Q$-functions.

\begin{remark}
The condition \eqref{eq:rv1} is an artifact of analyzing the differences between the $Q$-functions from the initial time $t=0$ to $t=-T$. If one instead carries out the same analysis using the state measures, one requires
\[
\frac{\sqrt{K_{\mathrm D} +\frac{\bar LK_1}{\rho}}}{\sqrt{K_{\mathrm D} +\frac{\bar LK_1}{\rho}}+\sqrt{\frac{\bar LK_1}{\rho}}} < \beta \text{ and } \sqrt{\beta}\left( \sqrt{K_{\mathrm D}+ \frac{\bar LK_1}{\rho}}+\sqrt{\frac{\bar LK_1}{\rho}} \right)<1
\]
to obtain the uniqueness of TRMFE. This is a stronger condition than \eqref{eq:rv1} and is proved in \cite[Theorem 5]{ayd} when \(Q_{\mathrm{terminal}} \equiv 0\).
\end{remark}

\begin{remark}
{Our arguments derive contraction conditions from the spectral radii of contraction matrices whose eigenvectors exhibit either exponential growth or exponential decay. These behaviors can be characterized by phase-transition conditions involving the discount factor, which differ depending on whether the iterations are formulated in terms of state measures or \(Q\)-functions. In particular,} if one attempts to establish the uniqueness of an infinite-horizon non-stationary MFE via $Q$-function iterations, the required uniqueness condition is
\(
\frac{\bar L K_1}{\rho}>\beta(1-K_{\mathrm D})^2,
\)
which is the exact opposite of the condition required in Theorem \ref{thrm:b}. However, this phase-shift condition differs from the one obtained through iterations of the state distributions; see \cite[Theorem 4]{ayd}.
\end{remark}

Note that the uniqueness condition in Theorem \ref{thrm:b} requires a sufficiently large discount factor. For our algorithmic purposes, where we use the TRMFE structure as an intermediate step to approximate a stationary MFE, we instead prefer a small discount factor or an equivalent condition.

\begin{theorem}\label{thrm:3}
    Suppose that Assumption \ref{ass:1} holds. Assume that
    \[
    \sqrt{\beta}\left(
    \sqrt{K_{\mathrm D}+\frac{\bar LK_1}{\rho}}
    +\sqrt{\frac{\bar LK_1}{\rho}}
    \right)<1
    \]
    and that either
    \[
    K_{\mathrm D}\geq1,
    \qquad\text{or}\qquad
    \frac{\bar L K_1}{\rho}>\beta(1-K_{\mathrm D})^2.
    \]
    Then, for any $T \in \mathbb N$ and $\mu_0 \in \mathcal P(X)$, there exists a unique $(\pi^T_{-t},\mu^T_{-t})_{t=0}^{T}\in \mathrm{MFE}_{T,\mu^T_0,Q^T_0}$ such that $Q^T_0=H_1(Q^T_0,\mu^T_0)$. Furthermore, this \(Q^T_0\) can be computed through a contractive mapping iteratively.
\end{theorem}
\begin{proof}
    We defer the proof to Section \ref{sect:one-shot}. Iterations that result in finding the invariant \(Q\)-function is described in Algorithm \ref{alg:0} below.
\end{proof}

If there exists a unique TRMFE, then every accumulation point of the tails of the finite-horizon MFEs obtained from Theorem \ref{thrm:3} as $T \to \infty$ induces a stationary MFE:

\begin{theorem}\label{thrm:suyo}
    Suppose that Assumption \ref{ass:1} holds. For all $T\in \mathbb N$, let $(\pi^T_{-t},\mu^T_{-t}) \in \mathrm{MFE}_{T,\mu_0,Q^T_0}$ such that $Q^T_0 = H_1(Q^T_0,\mu^T_0)$. Let $Q_*$ be an accumulation point of $(Q^T_0)_T$ in $(\mathcal C_{\beta},\|\cdot\|_{\infty})$. If for any terminal $Q$-function $Q_{\mathrm{terminal}}$ there exists a unique $\mathrm{TRMFE}_{Q_{\mathrm{terminal}}}$, then any accumulation point of $(\pi^T_0\otimes \mu^T_0)_T \subset \mathcal P(X \times \mathcal P(A))$ obtained under a sequence $(T_n)_{n \in \mathbb N}\subset \mathbb N$ such that $\lim_{n \to \infty} Q^{T_n}_0 = Q_*$ in $(\mathcal C_{\beta}, \| \cdot \|_\infty)$ induces a stationary MFE after disintegration.
\end{theorem}
\begin{proof}
By Theorem \ref{thrm:a}, every accumulation point of the family of TRMFEs
\(
\{(\pi^T_{-t} \otimes \mu^T_{-t})_{t=0}^T:T \in \mathbb N\}
\)
(after extending them to the infinite-horizon setting via concatenation), under the product topology on \(\prod_{T=0}^{\infty} \mathcal P(X\times A)\), where each factor is endowed with the topology of weak convergence, induces an infinite-horizon TRMFE under the terminal \(Q\)-function \(Q_*\). Note that if \((Q_{-t},\mu_{-t})_{t=0}^{\infty}\) induces a TRMFE, then \((Q_{-t},\mu_{-t})_{t=1}^{\infty}\) also induces a TRMFE due to the uniqueness of the minimizers of the \(Q\)-functions. However, {$Q^T_0=H_1(Q^T_0,\mu^T_0)$ for all $T$ implies \(Q_0=Q_{-1}\); thus,} the uniqueness of the TRMFE together with a straightforward induction argument implies that
\(
\pi_{-t}\otimes \mu_{-t} = \pi_{0} \otimes \mu_0
\)
for all \(t \in \mathbb N\).
\end{proof}

\begin{remark}
The condition
\(
\sqrt{\beta}\left( \sqrt{K_{\mathrm D}+ \frac{\bar LK_1}{\rho}}+\sqrt{\frac{\bar LK_1}{\rho}} \right)<1
\)
is the horizon-length-independent contraction condition established for regularized finite-horizon MFGs in \cite[Theorem~2]{ayd}. The contraction condition for stationary MFGs established in \cite{anahtarci2023q} is
\(
K_{\mathrm D} + \frac{K_1}{\rho}\frac{\bar L}{1-\beta}<1,
\)
which is often more restrictive than the finite-horizon MFG contraction condition. It is easy to verify that the Lipschitz parameter restrictions in Theorem~\ref{thrm:suyo} can hold even when
\(
K_{\mathrm D} + \frac{K_1}{\rho}\frac{\bar L}{1-\beta}>1.
\)
\end{remark}

Under the conditions of Theorem \ref{thrm:3}, the finite-horizon TRMFG is contractive; thus, we can always compute a finite-horizon MFE exponentially fast via iterations. Consequently, approximating a stationary MFE in the setting of Theorem \ref{thrm:3} amounts to computing a finite-horizon MFE and then identifying the asymptotic behavior of the terms that are sufficiently far from the terminal stage. However, to pass from a finite-horizon MFE to a stationary MFE, we require the uniqueness of the TRMFE under a given terminal $Q$-function. We emphasize that these conditions do not necessarily imply the existence of a stationary MFE.

\begin{algorithm}[h]\label{alg:0}
\caption{Finite-Horizon Fixed-Point with One-Shot Terminal Invariance}
    Fix initial state-measure $\mu_0 \in \mathcal P(X)$\;
    Initialize with $\pmb Q^T_0 = (Q_{0,t})_{t=0}^T \in \prod_{t=0}^T \mathcal C_{\beta}$\;
    \While{$\pmb Q^{T}_{n+1} \not = \pmb Q^{T}_n$}{
        \For{$t=0,\cdots,T,T+1$}{
            $\mu_{t+1}= H_2(Q_{n,t},\mu_{t})$
        }
         \For{$t=T-1,T-2,\cdots,0$}{
            $Q_{n+1,t} = H_1(Q_{n,t+1},\mu_{t})$\;
        }
        $Q_{n+1,T} = H_1(Q_{n,T},\mu_{T+1})$\;
        $\pmb Q^T_{n+1} = (Q_{n+1,t})_{t=0}^T$\;
    }
    Find an accumulation point $(\pmb \pi^*,(\mu^{*,0},\pmb \mu^*))$ obtained from the disintegration of weak-limit accumulation point of $\pmb \pi^{*,T}\otimes \pmb \mu^{*,T}$ as $T \to \infty$\;
\Return{stationary mean-field equilibrium $(\pmb \pi^*,\pmb \mu^*)$}
\end{algorithm}

However, under additional restrictions on the dynamics of the MFG, these results remain applicable even for large discount factors. In particular, a minorization condition together with a geometric drift condition having a sufficiently small drift factor $\alpha$ replaces the term $\beta$ in the relevant expressions with $\alpha\beta$. Consequently, even when the discount factor is sufficiently large, the results above remain applicable provided that the drift factor $\alpha$ is sufficiently small.

{For the remainder of this section, we work under the average-cost criterion and extend the preceding results to this setting.}
Let $W=\max_{x,a} \omega(x,a)$. Incorporating the weight function into the contraction argument multiplies the term $\frac{\bar LK_1}{\rho}$ by the factor $W$ (see, for instance, \cite[Theorem 3]{anahtarci2020value}). Let $\tilde P$ be the lifted transition sub-probability kernel obtained under Assumption \ref{ass:acmfg} (see Definition \ref{lem:-1}). Then the following result holds.
\begin{theorem}
Suppose that Assumptions \ref{ass:acmfg} and \ref{ass:1} hold. Then the following results hold.
\begin{enumerate}
    \item For any terminal \(Q\)-function \(Q_{\mathrm{terminal}}\in \mathcal C_\alpha\),
there exists a TRMFE for the average-cost MFTVP
with terminal condition \(Q_{\mathrm{terminal}}\).
    \item Suppose that for any given $Q_{\mathrm{terminal}}$ there exists a unique average-cost MFTVP. If \(Q^T_{\mathrm{terminal}}\to Q_{\mathrm{terminal}}\) in \(\mathcal C_\alpha\),
and \((\pi^T_{-t},\mu^T_{-t})_{t=0}^T\) is a sequence of finite-horizon MFEs obtained under
terminal \(Q\)-functions \(Q^T_{\mathrm{terminal}}\in \tilde{\mathcal C}_{\alpha}\) in the sub-MFTVP \((X,\mathcal P(A),C+\rho\Omega,\tilde P,Q^{T}_{\mathrm{terminal}})\), then every accumulation point of
\((\pi^T_{-t}\otimes \mu^T_{-t})_{t=0}^T\) in \(\prod_{T=0}^{\infty} \mathcal P(X\times A)\) (where each section is endowed with the weak convergence) induces an average-cost TRMFE equilibrium under \(Q_{\mathrm{terminal}}\).
\end{enumerate}
\end{theorem}
\begin{proof}
The proofs are identical to those of Theorem \ref{thrm:a}. The differences due to the weight function can be handled as in \cite[Theorem 3]{anahtarci2020value} by switching to the asymptotic optimality equation \eqref{eq:const8} instead of using the Bellman operator since \((\pi^T_{-t},\mu^T_{-t})_{t=0}^T\) are obtained under the sub-MFTVP.
\end{proof}

In general, it is often desirable to have a larger discount factor in infinite-horizon problems. Our final result provides an analogue of the contraction condition in Theorem \ref{thrm:3} for the finite-horizon problem under Assumption \ref{ass:acmfg}, which mitigates the effect of the discount factor.

\begin{theorem}\label{thrm:avg}
    Suppose that Assumptions \ref{ass:acmfg} and \ref{ass:1} hold. Let \(W:=\sup_{(x,a) \in X\times A}\omega(x,a)\) and $\beta\in(0,1)$.
If it holds that 
\[
\sqrt{\alpha\beta}\left( \sqrt{K_{\mathrm D}+ \frac{\bar LK_1}{\rho}W}+\sqrt{\frac{\bar LK_1}{\rho}W}\right)<1, \text{ and } \frac{\bar L K_1}{\rho}W>\alpha\beta(1-K_{\mathrm D})^2 \text{ or } K_D>1.
\]
then for any $T \in \mathbb N$ and $\mu_0 \in \mathcal P(X)$, there exists a unique $(\pi^T_{-t},\mu^T_{-t})_{t=0}^{T}\in \mathrm{MFE}_{T,\mu^T_0,Q^T_0}$ such that $Q^T_0=H_1^{\mathrm{avg}}(Q^T_0,\mu^T_0)$. Furthermore, this \(Q^T_0 \in \mathcal C_{\alpha\beta}\) can be computed through a contractive mapping iteratively.
\end{theorem}
\begin{proof}
    See Section~\ref{sect:average-cost-contraction}.
\end{proof}

\section{A Numerical Example}\label{sec:numerical-example2}

{
We provide a two-state, two-action example in which the infinite-horizon TRMFE is
unique for every prescribed terminal $Q$-function and the terminal-value
invariance iteration converges, whereas stationary best-response and
ordinary value iterations do not converge to a stationary MFE. 

We specify our unregularized MFG $(X,A,c,p)$ as follows. The state and action spaces are $X=A=\{0,1\}$, and the cost function and transition probability kernel are defined as 
\begin{align}
p(y\mid x,a,\mu)&=
\begin{cases}
0.49+0.02a, \, y=1\\
1-(0.49+0.02\,a), \, y=0
\end{cases}
\\
c(x,a,\mu)&=0.32+\frac{2x-1}{160}z(\mu)-0.002a f(z(\mu)),
\label{eq:validated-model}
\end{align}
where 
\[
z(\mu)=100\left(\mu(1)-\frac12\right),\qquad
\chi(x)=\max\{-1,\min\{x,1\}\},\qquad
f(x)=-\frac{51}{50}\chi(x)+\frac15\chi(x)^3.
\]
We fix the discount factor as
$\beta=4/5$, the regularization parameter as $\rho=1/1000$, and set the regularizer as
\[
\Omega(u)=u^2+(1-u)^2-\frac12.
\]

\begin{proposition}
    Under the configuration above, the following properties hold.
    \begin{enumerate}
        \item The regularizer $\Omega:\mathcal P(A) \to \mathbb R$ is nonnegative and $1$-strongly convex with respect to the $1$-norm.
        \item The cost function $c:X\times A \times \mathcal P(X) \to \mathbb R$ satisfies $0.00586\leq c\leq 0.63414$.
        \item Assumption~\ref{ass:acmfg} holds with $\lambda =0.98$, $\omega \equiv 1$, and $\alpha =0.02.$
        \item Assumption~\ref{ass:1} holds with $L_1 = 0.625$, and $K_1 = 0.02$.
    \end{enumerate}
\end{proposition}
\begin{proof}
    The proofs are straightforward, so we omit the details.
\end{proof}

In particular, the constants in
Convention~\ref{conv:5} are
\[
\bar L=\frac{0.625}{0.99984},
\qquad
K_{\mathrm D}=0.05,
\qquad
\frac{\bar L K_1}{\rho}W\simeq12.502.
\]
Since $\omega \equiv 1$, $W=1$, substitution of these constants into Theorem~\ref{thrm:avg} gives
\begin{align*}
&\sqrt{\alpha\beta}\left(
\sqrt{K_{\mathrm D}+\frac{\bar L K_1}{\rho}W}
+\sqrt{\frac{\bar L K_1}{\rho}W}
\right)
\simeq 0.895392<1,
\\
&\frac{\bar L K_1}{\rho}W
\simeq12.502
>\alpha\beta(1-K_{\mathrm D})^2=0.01444.
\end{align*}
Thus, the inner-loop present in Algorithm~\ref{alg:0} converges.

The infinite-horizon TRMFE is also unique for every terminal $Q$-function.
Indeed, the $Q$-function $(Q_t)_t$ obtained, $Q_t=H^{\mathrm{avg}}_1(Q_{t+1},\mu_t),$ during the update satisfies
\[
\min_u Q_{t}(1,u)-\min_uQ_{t}(0,u)=\frac{z(\mu_t)}{80}.
\]
Consequently, for an infinite-horizon TRMFE and $t\leq-2$,
\[
z(\mu_{t+1})=\chi\bigl(f(z(\mu_t))-0.1\,z(\mu_{t+1})\bigr)
=\frac{f(z(\mu_t))}{1.1}.
\]
Since $f(0)=0$ and $\operatorname{Lip}(f/1.1)=51/55<1$, it follows that
$z_t=0$ for all $t\leq-1$. For any prescribed terminal $Q$-function,
the remaining population is uniquely determined by
\[
z_0=\chi\!\left(-8\bigl[\min_u Q_{0}(1,u)-\min_u Q_{0}(0,u)\bigr]\right).
\]
Backward recursion determines the entire $Q$-function flow, and the
regularizer determines its policies uniquely. The stationary MFE is also
unique, with $\mu_*=(1/2,1/2)$ and minimizing action $(u(0),u(1))=(1/2,1/2)$ at both
states.

For the numerical experiment, we set $T=20$ and
$\mu_{\mathrm{initial}}(1)=.508$. We generated $2000$ state-symmetric
initial $Q$-function flows whose two affine action coefficients were
independently and uniformly distributed on $[0,.20]$. Every initialization
converged to the same fixed point; after $65$ updates, the largest uniform
error over the $2000$ runs was less than $1.6\times10^{-14}$.
Figure~\ref{fig:numerical-example2-terminal-convergence} reports the sample
mean together with one standard deviation.

\begin{figure}[H]
\centering
\includegraphics[width=.6\textwidth]{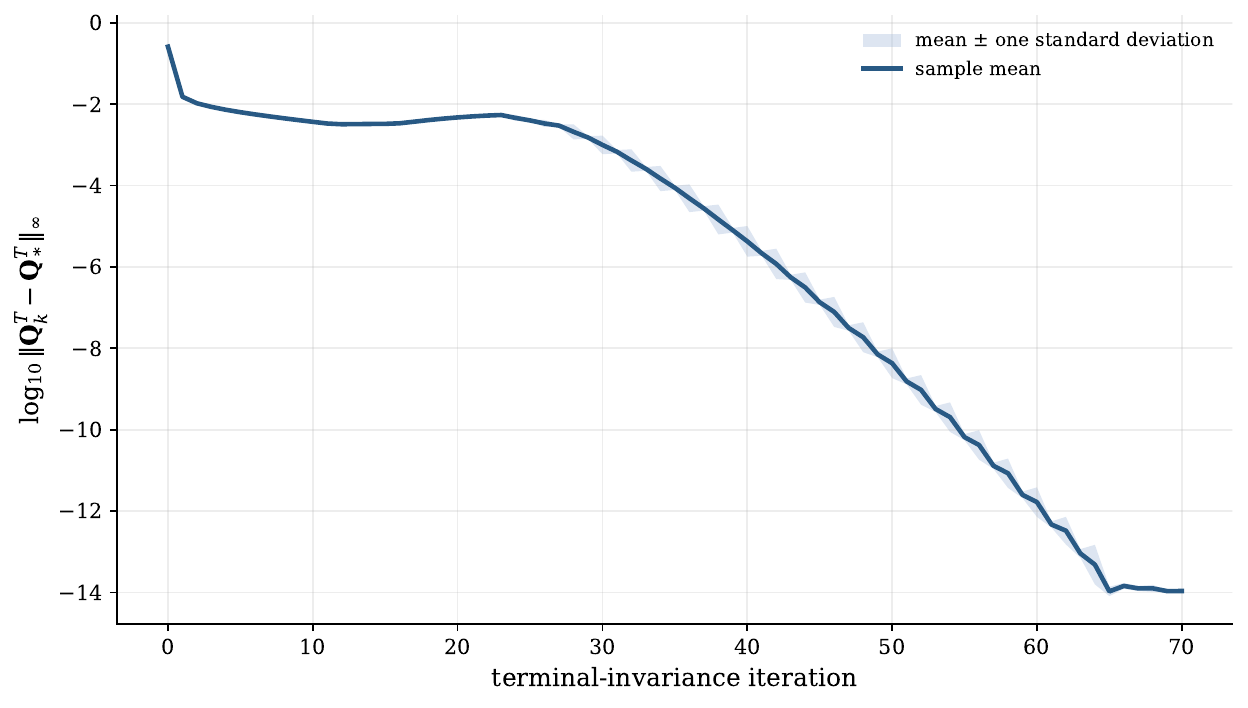}
\caption{Convergence of the discounted terminal-value invariance iteration
for $T=20$ under $2000$ independent random initializations.}
\label{fig:numerical-example2-terminal-convergence}
\end{figure}

We compare in Figure \ref{fig:numerical-example2-algorithm-comparison} our iteration with stationary best response, which solves
$Q_n=H_1^{\mathrm{avg}}(Q_n,\mu_n)$ and then updates $\mu_n$ using $H_2$,
and with ordinary value iteration, which updates $Q_n$ and $\mu_n$
simultaneously.
\begin{figure}[H]
\centering
\includegraphics[width=.8\textwidth]{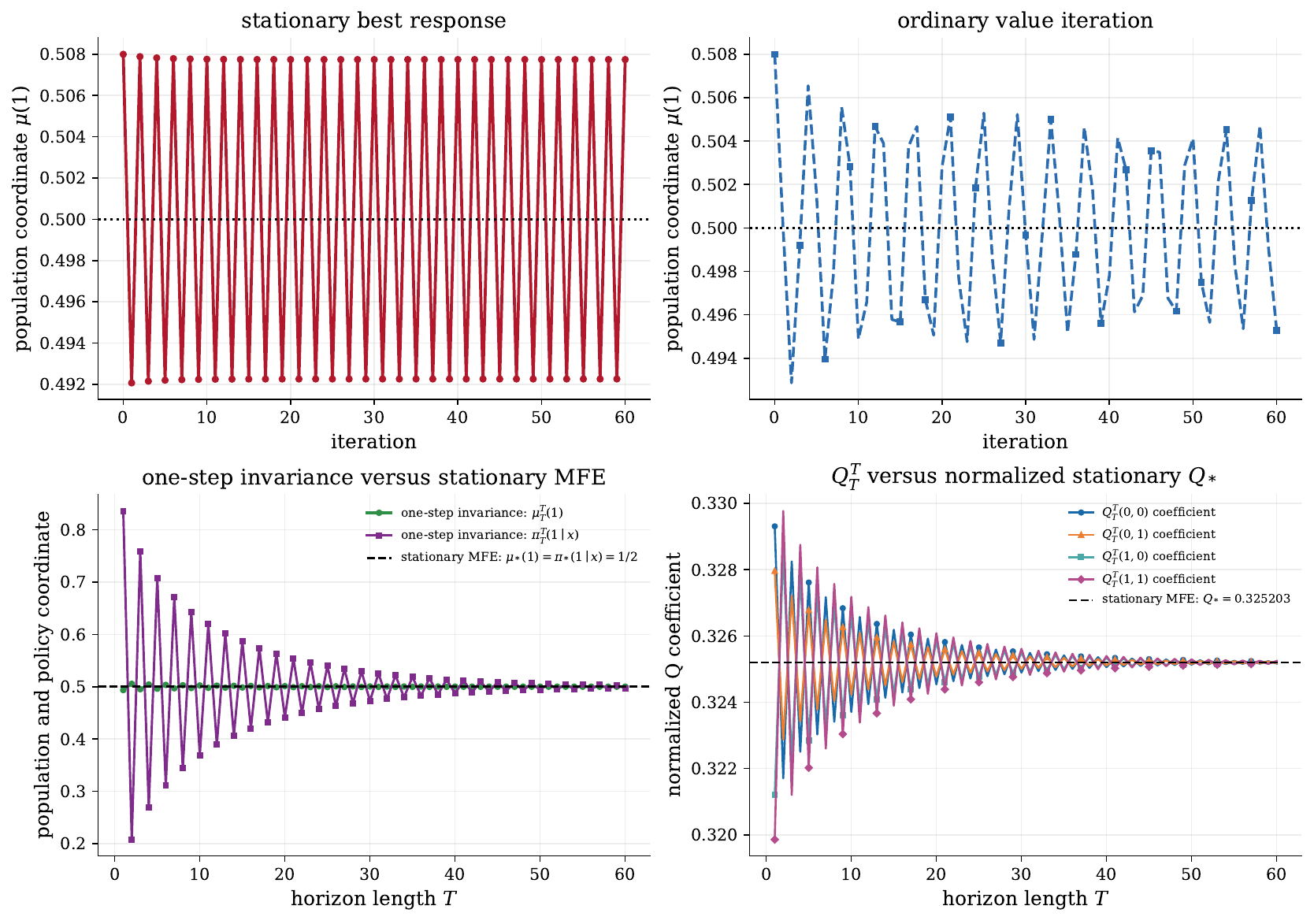}
\caption{Stationary best response approaches the two-cycle
$.4922540333$ and $.5077459667$, while ordinary value iteration remains
nonconvergent. Starting from $\mu_{\mathrm{initial}}(1)=.508$, the last two
panels show that the terminal population, terminal policy, and four affine
coefficients of $Q_T^T$ obtained from one-step invariance converge to the
corresponding stationary-MFE values as $T$ increases.}
\label{fig:numerical-example2-algorithm-comparison}
\end{figure}
}

\printbibliography
\section{Properties of $H_1$ and $H_2$ under Regularization}\label{app:a}
In this appendix, we state several results related to \(H_1\) and \(H_2\). The proofs can be found in \cite{anahtarci2023q}. We define $Q_{\min} = \min_u Q(\cdot,u)$ and $w_{\max}(x) = \max_u w(x,u)$. The pseudo norm $\| Q_{\min} \|_{\mathrm{Lip}}$ refers to the Lipchitz constant of $Q_{\min}$.
\subsection{Discounted Case}
\begin{lemma}
There exists \(M>0\) such that for any $Q \in \mathcal C_{\beta}$ we have
    \[
    \sup_{\substack{x \in X\\ u \in \mathcal P(A)}}|Q(x,u)-\rho\Omega(u)| \le \frac{M}{1-\beta}
    \text{ and }
    \sup_{\substack{x,\tilde x \in X,\\ u\in \mathcal P(A)}}| Q(x,u)-Q(\tilde x,\tilde u)| \le \frac{L_1}{1-\frac{\beta K_1}{2}} 1_{\{x \not = \tilde x\}}
    \]
    for all $x,\tilde x \in X$.
\end{lemma}
\begin{proof}
    Directly follows from Assumption \ref{ass:1}.
\end{proof}

\begin{lemma}\label{lem:2}
    Let \(\mathfrak q :X \times A \to \mathbb R\). Define a \(Q\)-function as $Q(\cdot,u) \equiv \sum _{a \in A} \mathfrak q(\cdot,a) u(a) + \rho\Omega(u)$ for all $u \in \mathcal P(A)$. If $\| Q-\rho\Omega\|_{\infty} \le \frac{M}{1-\beta}$ and 
    \[
    \sup_{\substack{x,\tilde x \in X,\\ u,\tilde u\in \mathcal P(A)}}| Q(x,u)-\rho\Omega(u)-Q(\tilde x,\tilde u)+\rho\Omega(\tilde u)| \le \frac{L_1}{1-\frac{\beta K_1}{2}}1_{\{x \not = \tilde x\}},
    \]
    then $\mathfrak q \in \tilde{\mathcal C}_{\beta}.$
\end{lemma}
\begin{proof}
    Directly follows from the definitions.
\end{proof}

The lemmas above establishes a one-to-one correspondence between $\mathcal C_{\beta}$ and $\tilde{\mathcal C}_{\beta}$. It is obvious that $(\tilde{\mathcal C}_{\beta},\| \cdot\|_{\infty})$ is compact. The lemmas above imply that the identity map between $(\tilde{\mathcal C}_{\beta},\| \cdot\|_{\infty})$ and $(\mathcal C_{\beta},\| \cdot\|_{\infty})$ is continuous and a bijection. Thus, $(\mathcal C_{\beta},\| \cdot\|_{\infty})$ is also compact.
\begin{lemma}\label{lem:3}
    For all $1>\beta>0,$ the space $(\mathcal C_{\beta},\| \cdot\|_{\infty})$ is compact.
\end{lemma}
\begin{proof}
    By the Arzela-Ascoli theorem, the set \(\mathcal {\tilde C}_{\beta}\) is compact under the uniform norm. The result then follows from Lemma \ref{lem:2}.
\end{proof}

\begin{lemma}\label{lem:4}
    Suppose that Assumption \ref{ass:1} holds. For a fixed $\mu \in \mathcal P(X)$, the operator $H_1(\cdot,\mu)$ maps $\mathcal C_{\beta}$ to itself.
\end{lemma}
\begin{proof}
    This result can be proved using the same lines of proof as in \cite[Lemma 1]{anahtarci2023q}.
\end{proof}

\begin{lemma}\label{lem:as}
    Suppose that Assumption \ref{ass:1} holds. Then, for all \( x, \tilde{x} \in X \), \( u, \tilde{u} \in \mathcal P(A) \), and \( \mu, \tilde{\mu} \in \mathcal P(X) \), $C$ and $P$ satisfy the following Lipschitz bounds:
\[
|C(x, u, \mu) - C(\tilde{x}, \tilde{u}, \tilde{\mu})| \leq L_1
\left( \mathbf{1}_{\{x \neq \tilde{x}\}} + \|u - \tilde{u}\|_1 + \|\mu - \tilde{\mu}\|_{1} \right),
\]

\[
\|P(\cdot \mid x, u, \mu) - P(\cdot \mid \tilde{x}, \tilde{u}, \tilde{\mu})\|_{1} \leq K_1
\left( \mathbf{1}_{\{x \neq \tilde{x}\}} + \|u - \tilde{u}\|_1 + \|\mu - \tilde{\mu}\|_{1} \right).
\]
\end{lemma}
\begin{proof}
    This result directly follows from \cite[Proposition 1]{anahtarci2023q}.
\end{proof}

\begin{lemma}\label{lem:7}
Suppose that Assumption \ref{ass:1} holds. Let $Q^1,Q^2\in \mathcal C_{\beta}$ and $\mu_1,\mu_2 \in \mathcal P(X)$. By \(\| Q_{\min} \|_{\mathrm{Lip}}\), denote the Lipschitz coefficent of \(Q_{\min}\) under the norm \( \| \cdot \|_{\infty}\). Then, we have
\[
\|H_1(Q^1,\mu_1) - H_1(Q^2,\mu_2)\|_{\infty}  \le  K_1\| Q^1_{\min} \|_{\mathrm{Lip}}\| \mu_1 - \mu_2\|_1 + \beta \| Q^1-Q^2\|_{\infty}.
\]
\end{lemma}
\begin{proof}
    This result can be proved using the same lines of proof as in \cite[Lemma 2]{anahtarci2023q}.
\end{proof}

\begin{lemma}\label{lem:8}
    Suppose that Assumption \ref{ass:1} holds. Let $\mu,\tilde \mu \in \mathcal P(X)$ and $Q,\tilde Q \in \mathcal C_{\beta}$. Then, we have
    \[
    \| H_{2}(Q,\mu) - H_{2}(\tilde Q,\tilde \mu) \|_1
    \leq \frac {K_1}{\rho}\| Q - \tilde Q \|_{\infty}
    +K_{\mathrm D}\|\mu-\tilde \mu\|_1.
    \]
\end{lemma}

\begin{proof}
For $Q\in\mathcal C_\beta$, let $u_Q(x):=\underset{u\in\mathcal P(A)}{\operatorname{argmin}}\,Q(x,u).$
By \cite[Lemma 3]{anahtarci2023q}, the strong convexity estimate for the
minimizers gives
\begin{equation}\label{eq:h2-policy-sensitivity}
\sup_{x\in X}\|u_Q(x)-u_{\tilde Q}(x)\|_1
\leq \frac{1}{\rho}\|Q-\tilde Q\|_\infty.
\end{equation}

For fixed $(Q,\mu)$, Lemma \ref{lem:as} and the elementary estimate
$\|u_Q(x)-u_Q(\tilde x)\|_1\leq 2$ imply
\[
\sup_{x,\tilde x\in X}
\frac{1}{2}
\left\|
P(\cdot\mid x,u_Q(x),\mu)
-
P(\cdot\mid \tilde x,u_Q(\tilde x),\mu)
\right\|_1
\leq \frac{3}{2}K_1.
\]
Since the left-hand side is also bounded above by $1$, we have
\[
\sup_{x,\tilde x\in X}
\frac{1}{2}
\left\|
P(\cdot\mid x,u_Q(x),\mu)
-
P(\cdot\mid \tilde x,u_Q(\tilde x),\mu)
\right\|_1
\leq
\min\left\{1,\frac{3}{2}K_1\right\}.
\]

By the definition of $H_2$, for every $y\in X$,
\[
H_2(Q,\mu)(y) = \sum_{x\in X}
P(y\mid x,u_Q(x),\mu)\, \mu(x).
\]
Therefore,
\begin{align*}
\|H_2(Q,\mu)-H_2(\tilde Q,\tilde\mu)\|_1
&\leq
\left\|
\sum_{x\in X}
P(\cdot\mid x,u_Q(x),\mu)(\mu(x)-\tilde\mu(x))
\right\|_1
\\
&\quad+
\left\|
\sum_{x\in X}
\Bigl(
P(\cdot\mid x,u_Q(x),\mu)
-
P(\cdot\mid x,u_Q(x),\tilde\mu)
\Bigr)\tilde\mu(x)
\right\|_1
\\
&\quad+
\left\|
\sum_{x\in X}
\Bigl(
P(\cdot\mid x,u_Q(x),\tilde\mu)
-
P(\cdot\mid x,u_{\tilde Q}(x),\tilde\mu)
\Bigr)\tilde\mu(x)
\right\|_1.
\end{align*}

Repeating the argument in
\cite[Equations (1.12), (1.5), and (1.19)]{dobrushin}, we have
\begin{align*}
&\left\|
\sum_{x\in X}
(\mu(x)-\tilde\mu(x))
P(\cdot\mid x,u_Q(x),\mu)
\right\|_1
\\
&=
\frac{1}{2}\|\mu-\tilde\mu\|_1
\left\|
\sum_{x,\tilde x\in X}
\frac{2(\mu(x)-\tilde\mu(x))_+}
{\|\mu-\tilde\mu\|_1}
\frac{2(\tilde\mu(\tilde x)-\mu(\tilde x))_+}
{\|\mu-\tilde\mu\|_1}
\,
\Bigl(
P(\cdot\mid x,u_Q(x),\mu)
-
P(\cdot\mid \tilde x,u_Q(\tilde x),\mu)
\Bigr)
\right\|_1
\\
&\leq
\frac{1}{2}\|\mu-\tilde\mu\|_1
\sum_{x,\tilde x\in X}
\frac{2(\mu(x)-\tilde\mu(x))_+}
{\|\mu-\tilde\mu\|_1}
\frac{2(\tilde\mu(\tilde x)-\mu(\tilde x))_+}
{\|\mu-\tilde\mu\|_1}
\,
\left\|
P(\cdot\mid x,u_Q(x),\mu)
-
P(\cdot\mid \tilde x,u_Q(\tilde x),\mu)
\right\|_1
\\
&\leq
\frac{1}{2}\|\mu-\tilde\mu\|_1
\sum_{x,\tilde x\in X}
\frac{2(\mu(x)-\tilde\mu(x))_+}
{\|\mu-\tilde\mu\|_1}
\frac{2(\tilde\mu(\tilde x)-\mu(\tilde x))_+}
{\|\mu-\tilde\mu\|_1}
\,
2\min\left\{1,\frac{3}{2}K_1\right\}
\\
&=
\min\left\{1,\frac{3}{2}K_1\right\}
\|\mu-\tilde\mu\|_1,
\end{align*}
where the last line follows from
$\sum_{x\in X}\bigl(\mu(x)-\tilde\mu(x)\bigr)=0,$
which implies
\[
\sum_{x\in X}\underbrace{\max\{\mu(x)-\tilde\mu(x),0\}}_{=(\mu(x)-\tilde \mu(x))_+}
=
\sum_{x\in X}\underbrace{\max\{\tilde\mu(x)-\mu(x),0\}}_{= (\mu(x)-\tilde \mu(x))_-}
=
\frac{1}{2}\|\mu-\tilde\mu\|_1.
\]

For the second term, the Lipschitz continuity of $P$ gives
\begin{align*}
&\left\|
\sum_{x\in X}
\tilde\mu(x)
\Bigl(
P(\cdot\mid x,u_Q(x),\mu)
-
P(\cdot\mid x,u_Q(x),\tilde\mu)
\Bigr)
\right\|_1
\\
&\leq
\sum_{x\in X}
\tilde\mu(x)
\left\|
P(\cdot\mid x,u_Q(x),\mu)
-
P(\cdot\mid x,u_Q(x),\tilde\mu)
\right\|_1
\\
&\leq
K_1\|\mu-\tilde\mu\|_1.
\end{align*}

Similarly, for the third term,
\begin{align*}
\left\|
\sum_{x\in X}
\tilde\mu(x)
\Bigl(
P(\cdot\mid x,u_Q(x),\tilde\mu)
-
P(\cdot\mid x,u_{\tilde Q}(x),\tilde\mu)
\Bigr)
\right\|_1
&\leq
K_1
\sum_{x\in X}
\tilde\mu(x)
\|u_Q(x)-u_{\tilde Q}(x)\|_1
\\
&\leq
\frac{K_1}{\rho}\|Q-\tilde Q\|_\infty,
\end{align*}
where the last inequality follows from
\eqref{eq:h2-policy-sensitivity}.

Combining the three estimates yields
\[
\|H_2(Q,\mu)-H_2(\tilde Q,\tilde\mu)\|_1
\leq
\left(
K_1+\min\left\{1,\frac{3}{2}K_1\right\}
\right)
\|\mu-\tilde\mu\|_1
+
\frac{K_1}{\rho}\|Q-\tilde Q\|_\infty.
\]
The definition of $K_{\mathrm D}$ now proves the claim.
\end{proof}
\subsection{The minorized discounted-cost case}
The following lemma is the weighted version of Lemmas \ref{lem:7} and \ref{lem:8} which we will use in the proof of Theorem \ref{thrm:avg}.

\begin{lemma}\label{lem:9}
    \begin{enumerate}
        \item []
        \item  For all \(\mu,\tilde \mu \in \mathcal P(X)\) and \(Q,\tilde Q \in \mathcal C_{\alpha\beta}\), it holds that
    \begin{equation}\label{eq:quaso1}
    \| H_{2}(Q,\mu) - H_{2}(\tilde Q,\tilde \mu) \|_1
    \leq \frac {K_1}{\rho}W\| Q - \tilde Q \|_{\omega}
    +K_{\mathrm D}\|\mu-\tilde \mu\|_1.
    \end{equation}
    \item For all \(\mu_1,\mu_2 \in \mathcal P(X)\) and \(Q^1,Q^2 \in \mathcal C_{\alpha\beta}\), it holds that
    \begin{equation}\label{eq:quaso2}
\|H^{\mathrm{avg}}_1(Q^1,\mu_1) - H_1^{\mathrm{avg}}(Q^2,\mu_2)\|_{\omega}
\leq \bar L\| \mu_1 - \mu_2\|_1 + \alpha\beta \| Q^1-Q^2\|_{\omega}.
    \end{equation}
    \end{enumerate}
   
\end{lemma}
\begin{proof}
Recall that \(W := \sup_{(x,a) \in X\times A}\omega(x,a)\). Since
\[
\|Q-\tilde Q\|_\infty\leq W\|Q-\tilde Q\|_\omega,
\]
the proof of Lemma~\ref{lem:8} gives \eqref{eq:quaso1}.

For \eqref{eq:quaso2}, add and subtract
\(
\sum_{y\in X}Q^1_{\min}(y)\tilde P(y\mid x,u,\mu_2)
\)
in the difference of the two operators. Then,
\begin{align*}
&\left|C(x,u,\mu_1)-C(x,u,\mu_2)
+\sum_{y\in X}Q^1_{\min}(y)
\left(\tilde P(y\mid x,u,\mu_1)-\tilde P(y\mid x,u,\mu_2)\right)\right|
\\&\leq \bar L\|\mu_1-\mu_2\|_1\omega(x,u).
\end{align*}
Moreover, Assumption~\ref{ass:acmfg} and Convention~\ref{conv} give
\begin{align*}
\left|\beta\sum_{y\in X}
\left(Q^1_{\min}(y)-Q^2_{\min}(y)\right)
\tilde P(y\mid x,u,\mu_2)\right|
&\leq
\beta\|Q^1-Q^2\|_\omega
\sum_{y\in X}\omega_{\max}(y)\tilde P(y\mid x,u,\mu_2)
\\&\leq
\alpha\beta\|Q^1-Q^2\|_\omega\omega(x,u).
\end{align*}
Dividing by \(\omega(x,u)\) and taking the supremum proves \eqref{eq:quaso2}.
\end{proof}

We note that, for all \(\mu \in \mathcal P(X)\), \(H_1^{\mathrm{avg}}(\cdot,\mu)\) is invariant over \(\mathcal C_{\alpha\beta}\), which can be proved as in Lemma \ref{lem:4}.

\section{Existence and Uniqueness of Infinite-Horizon Time-Reversed Mean-field Equilibria under Fixed Terminal Cost}\label{sect:exist}
In this section, we provide existence and uniqueness criteria for time-reversed mean-field equilibria under a fixed terminal cost, which amounts to the proofs of Theorems \ref{thrm:a} and \ref{thrm:b}. In particular, we study the following problems:

\begin{enumerate}
    \item Given a terminal $Q$-function in \(\mathcal C_{\beta}\), we establish the existence of a TRMFE. Let $(Q_n)_n$ be a family of terminal $Q$-functions converging to some $Q_*$. We show that every accumulation point of the finite-horizon MFEs with horizon length $n$, obtained under the terminal $Q$-function \(Q_n\), belongs to \(\mathrm{TRMFE}_{Q_*}\). This amounts to the proof of Theorem \ref{thrm:a}.

    \item We prove Theorem \ref{thrm:b}, which provides a uniqueness condition for \(\mathrm{TRMFE}_{Q_*}\).
\end{enumerate}

For simplicity, we present these results for the discounted case. However, the proofs extend to the average-cost case with only minimal modifications.

\subsection{Existence of Infinite-Horizon Time-Reversed Mean-field Equilibria}\label{sect:4.1}

In the case of infinite-horizon non-stationary MFGs, the existence of an MFE is established under mild assumptions in \cite{SaBaRaSIAM}. However, for MFTVPs, we do not have access to a global structure such as an objective function. Thus, the techniques developed in \cite{SaBaRaSIAM} are not directly applicable to MFTVPs.

In this subsection, we establish the existence of a TRMFE in the infinite-horizon discounted setting under a fixed terminal $Q$-function in $\mathcal C_{\beta}$. Our approach relies on studying finite-horizon MFEs with the given terminal $Q$-function and a fixed initial state measure. We then show that, as the horizon length \(T\) tends to infinity, one obtains a TRMFE (up to the extraction of a subsequence). Since the proofs are essentially identical, we state a more general result in which the terminal $Q$-function is allowed to vary with the horizon length \(T\). This more general formulation is used later to establish the convergence of a class of finite-horizon MFEs to a stationary MFE.

The main result of this subsection is the following existence theorem, which establishes the first part of Theorem \ref{thrm:a}.

\begin{theorem}\label{thrm:1}
Suppose that Assumption \ref{ass:1} holds. Then, for any $Q_* \in \mathcal C_{\beta}$, there exists a solution to the MFTVP with terminal $Q$-function $Q_*$.
\end{theorem}

Our proof of Theorem \ref{thrm:1} has two objectives. First, as stated, we establish the existence of a TRMFE under a given terminal $Q$-function in a suitable domain. Second, the proof proceeds by approximating the TRMFE via finite-horizon MFEs obtained under terminal $Q$-functions, which is essential for our stationary MFE approximation scheme. For completeness, we state this convergence result separately for later reference, as it constitutes the second part of Theorem \ref{thrm:a}.

\begin{theorem}
Let \((\pi^T_t \otimes \mu^T_t)_{t=0}^{T-1} \in \mathrm{MFE}_{T,\mu_0,Q^T_{\mathrm{terminal}}}\) for each \(T \in \mathbb N\). Suppose that \(Q^T_{\mathrm{terminal}} \to Q_{\mathrm{terminal}}\) under the uniform norm. Then, every accumulation point of the sequence \(\{(\pi^T_t\otimes \mu^T_t)_{t=0}^{T-1}\}_{T=1}^{\infty}\) in \(\mathcal P(X \times \mathcal P(A))^{\infty}\) belongs to \(\mathrm{TRMFE}_{Q_{\mathrm{terminal}}}\).
\end{theorem}

Before proceeding with the proof of Theorem \ref{thrm:1}, we state several technical lemmas that are essential for our purposes. To pass from finite-horizon MFE to TRMFE, we extract convergent subsequences of the corresponding $Q$-functions. First, we show that $(\mathcal C_{\beta},\|\cdot\|_{\infty})$ is a compact invariant domain for \(H_1(\cdot,\mu)\). In particular, for a given sequence of finite-horizon MFEs, we can extract a convergent subsequence of the corresponding $Q$-functions at each time index \(t\) as the horizon length \(T \to \infty\).

\begin{lemma}
For any $\mu \in \mathcal P(X)$ and $Q \in \mathcal C_{\beta}$, we have \(H_1(Q,\mu) \in \mathcal C_{\beta}\). Furthermore, the set $\mathcal C_{\beta}$ is compact under the uniform norm.
\end{lemma}
\begin{proof}
The fact that for $Q \in \mathcal C_{\beta}$ we have \(H_1(Q,\mu) \in \mathcal C_{\beta}\) can be proved as in Lemma \ref{lem:4}. The second statement follows from Lemma \ref{lem:3}.
\end{proof}

The evolution of the state measures also requires us to control the joint probability measures constructed from the optimal policies and the state measures, since the limiting joint probability measures must also concentrate on the optimal state-action pairs of the limiting $Q$-functions. Next, we demonstrate that, as long as the $Q$-functions converge uniformly, this property is preserved.

\begin{lemma}\label{lem:1}
Let $(Q_n)_n$ be a family of uniformly bounded continuous real-valued functions over $X\times A$ such that $\lim_{n \to \infty} Q_n(x_n,a_n) = Q(x,a)$ for all $(x_n,a_n)_n,(x,a) \in X \times A$ such that $\lim_{n \to \infty} (x_n,a_n) = (x,a)$. Then, if $(\pi_n \otimes \mu_n)_{n \in \mathbb N} \subset \mathcal P(X \times A)$ concentrates on the optimal state-action pairs of the continuous function $Q_n : X \times A \to \mathbb R$ for all $n$, then the weak limit of $(\pi_n \otimes \mu_n)_n$ in $\mathcal P(X\times A)$ also concentrates on the optimal state-action pairs of the function $Q$, provided that the limit  of $(\pi_n \otimes \mu_n)_n$ exists.
\end{lemma}
\begin{proof}
This result is proved in \cite[Lemma 12]{ayd}.
\end{proof}

\begin{proof}[Proof of Theorem \ref{thrm:1}]
    Let $(Q^T_{-T},\cdots,Q^T_{-1},Q^T_0)$ be $Q$-functions and $(\mu^T_{-T}\cdots,\mu^T_{-1},\mu^T_0)$ be state-measures such that $(\pi_{Q^T_{-t}},\mu^T_{-t})_{t=0}^T$ induces a finite-horizon TRMFE with $\mu^T_{-T}=\mu_*$ and $Q^T_0=Q_*$ for all  for all $T \in \mathbb N$. Since $Q^T_{-t} \in \mathcal C_{\beta}$ for $t \in \{0,1,\cdots,T\}$, and $(\mathcal C_{\beta},\| \cdot\|_{\infty})$, and $(\mathcal P(X),\|\cdot\|_1)$ are compact spaces, for any given $t\in \mathbb N$, using a diagonalization argument, we can extract a sequence of horizon lengths $(T_n)_n$ such that $\lim_{n \to \infty} T_n = \infty$, $\lim_{n\to \infty} Q^{T_n}_{-t}=:Q_{-t}$ in $(\mathcal C_{\beta},\|\cdot\|_\infty)$, and $\lim_{n \to \infty} \mu^{T_n}_{-t}=:\mu_{-t}$ exists in $(\mathcal P(X),\|\cdot\|_1)$.

    First, for all $0 \le t \le T_n$, recall the recursion
    \[
    Q^{T_n}_{-t}(x,u) = C(x,u,\mu^{T_n}_{-t})+\rho\Omega(u)+\beta\sum_{y \in X}\min_{b \in \mathcal P(A)}Q^{T_n}_{-t+1}(y,b)\,P(y \mid x,u,\mu^{T_n}_{-t}).
    \]
    Since $\lim_{n\to \infty} Q^{T_n}_{-t}=:Q_{-t}$ under the uniform norm, and that $P(\cdot \mid \cdot,\cdot,\mu^{T_n}_{-t})$ converges continuously to $P(\cdot \mid \cdot,\cdot,\mu_{-t})$ in $(\mathcal P(X),\| \cdot \|_1)$, as a consequence of \cite[Theorem 3.5]{serfozo1982convergence}, it holds that
    \[
    \lim_{n \to \infty} \sum_{y \in X}\min_{b \in \mathcal P(A)}Q^{T_n}_{-t+1}(y,b)\,P(y \mid x,u,\mu^{T_n}_{-t}) = \sum_{y \in X}\min_{b \in \mathcal P(A)}Q_{-t+1}(y,b)\,P(y \mid x,u,\mu_{-t}).
    \]
    Thus, for all $t \in \mathbb N$, we have
    \[
    Q_{-t}(x,u) = C(x,u,\mu_{-t})+\rho\Omega(u)+\beta\sum_{y \in X}\min_{b \in \mathcal P(A)}Q_{-t+1}(y,b)\,P(y \mid x,u,\mu_{-t}).
    \]

    Furthermore, $\lim_{n\to \infty} Q^{T_n}_{-t}=:Q_{-t}$ under the uniform norm implies that
    \[
    \lim_{n \to \infty} \pi_{Q^{T_n}_{-t}}(\cdot\mid \cdot)\mu^{T_n}_{-t}(\cdot) = \pi_{Q_{-t}}(\cdot\mid \cdot)\mu_{-t}(\cdot) 
    \]
    on $\mathcal P(X\times\mathcal P(A))$ under the weak convergence (for instance, see the proof of \cite[Theorem 8]{ayd}). Thus, by \cite[Lemma 14]{ayd}, using the recursion
    \[
    \mu^{T_n}_{-t}(\cdot)=\sum_{x \in X}\int_{u \in \mathcal P(A)} P(\cdot \mid x,u,\mu^{T_n}_{-t-1})\,\pi^{T_n}_{-t-1}(u \mid x)\,\mu^{T_n}_{-t-1}(x),
    \]
    we obtain that 
    \[
    \mu_{-t}(\cdot)=\sum_{x \in X}\int_{u \in \mathcal P(A)} P(\cdot \mid x,u,\mu_{-t-1})\,\pi_{-t-1}(u \mid x)\,\mu_{-t-1}(x).
    \]
    By definition, $(\pi_{-t},\mu_{-t})_{t=0}^{\infty}$ is a TRMFE.
\end{proof}

\subsection{Uniqueness of Infinite-Horizon Time-Reversed Mean-field Equilibrium under Contraction}\label{sect:4.3}

In this subsection, we provide a uniqueness condition for time-reversed MFE under any terminal \(Q\)-function \(Q_{\mathrm{terminal}} \in \mathcal C_{\beta}\) via a contraction argument. We do not use this uniqueness result directly in our approximation scheme because the required contraction condition lies in the opposite regime from the contraction condition required for the one-step invariance iterations considered at the finite-horizon level. Nevertheless, the results of this subsection may be of independent interest and are also used in the proof of the terminal \(Q\)-function invariance condition stated in Theorem \ref{thrm:3}.

The uniqueness of an infinite-horizon MFTVP relies on studying the tail behavior of a finite-horizon MFE that starts from a fixed initial state measure and terminates with a prescribed terminal cost. To establish uniqueness via finite-time error bounds, we require two ingredients. The first is a contraction result for the finite-horizon problem under an arbitrary terminal cost and initial state measure, obtained through iterations of the \(Q\)-functions. This contraction result amounts to studying the spectral radius of a nonnegative irreducible matrix constructed from the Lipschitz parameters of \(H_1\) and \(H_2\). The second ingredient is an asymptotic characterization of the Perron right eigenvectors of this matrix.

Our contraction result relies on a \(Q\)-function-based variation of the iteration method for state measures developed in \cite[Theorem 2]{ayd}. The main difference here is that we incorporate terminal \(Q\)-functions into the analysis and employ a slightly different argument based on the Collatz--Wielandt bounds, rather than studying the roots of the characteristic polynomial of the contraction matrix. This approach bypasses several auxiliary results that rely on complex analysis; however, the techniques developed in \cite{ayd} remain applicable in our setting. Furthermore, since our goal is to analyze TRMFE, the orientation of the iterations is reversed: specifically, we consider iterations of the adjoint system. Nevertheless, the same asymptotic convergence properties established in \cite[Theorem 3]{ayd} continue to hold in our setting.

\begin{figure}[H]
\begin{tikzcd}
Q_{-T} \arrow[dr] \arrow[dr] & Q_{-T+1} \arrow[dr] \arrow[ddl, bend right=25] &Q_{-T+2} \arrow[dr] \arrow[ddl, bend right=25] & \cdots \arrow[dr] \arrow[ddl, bend right=25] & Q_{-1} \arrow[ddl, bend right=25] & \underbrace{Q_{\mathrm{terminal}}}_{\mathrm{fixed}} \arrow[ddl, bend right=25] &\mathrm{Input} \\
\mu_{-T}=\underbrace{\mu_{\mathrm{initial}}}_{\mathrm{fixed}} \arrow[r] \arrow[d] & \mu^{\pmb Q}_{-T+1} \arrow[r]\arrow[d] & \mu^{\pmb Q}_{-T+2} \arrow[r] \arrow[d] & \cdots \arrow[r] \arrow[d] & \mu^{\pmb Q}_{-1} \arrow[d] &  &\\
Q^{\mathrm{new}}_{-T} & Q^{\mathrm{new}}_{-T+1} & Q^{\mathrm{new}}_{-T+2} & \cdots & Q^{\mathrm{new}}_{-1} &   &\mathrm{Output} \arrow[uu, bend right=35]
\end{tikzcd}
\caption{Finite-horizon MFG iterative scheme under a fixed terminal $Q$-function \(Q_{\mathrm{terminal}}\) and an initial state measure \(\mu_{\mathrm{initial}}\).}
\label{fig:0}
\end{figure}
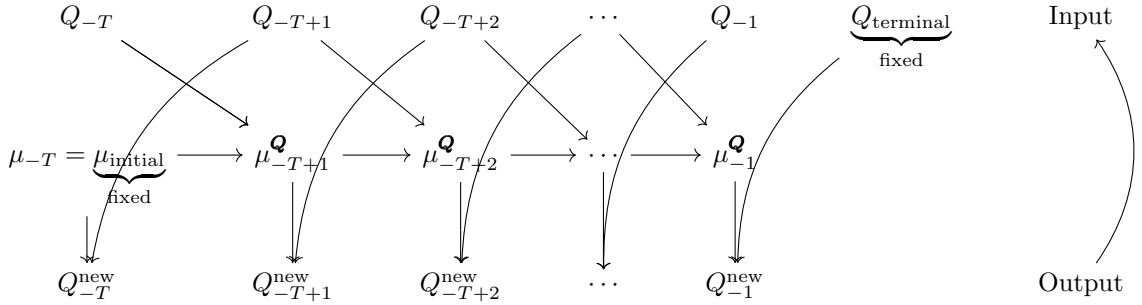

Below, we present an asymptotic characterization of the contraction properties of the iteration depicted in Figure \ref{fig:0}. Although an equivalent contraction condition was established in \cite[Theorem 3]{ayd}, we provide an alternative characterization based on more elementary tools. Furthermore, this variational formulation yields a simpler proof of the convergence rate between finite-horizon and infinite-horizon MFEs, which is the key ingredient in establishing uniqueness of the infinite-horizon TRMFE under any given terminal $Q$-function.

\begin{lemma}\label{lem:33333}
Suppose that Assumption \ref{ass:1} holds. Let \( \pmb Q = (Q_{-t})_{t=0}^T, \pmb {\hat Q}=(\hat Q_{-t})_{t=0}^T \subset \mathcal C_{\beta}\) and \(\mu_0 \in \mathcal P(X)\) be given. Define \(
\mu^{\pmb Q}_{-t+1} = H_2(Q_{-t},\mu^{\pmb Q}_{-t}),
\)
and 
\(
\mu^{\pmb {\hat Q}}_{-t+1} = H_2(\hat Q_{-t}, \mu_{-t}).
\)
We iteratively define new \(Q\)-functions as
\(
Q^{\pmb {\mu^Q}}_{-t} = H_1(Q^{\pmb {\mu^Q}}_{-t+1},\mu^{\pmb {Q}}_{-t})
\)
and
\(
Q^{\pmb {\mu^{\hat Q}}}_{-t} = H_1(Q^{\pmb {\mu^{\hat Q}}}_{-t+1},\mu^{\pmb {\hat Q}}_{-t}),
\)
where \(Q^{\pmb {\mu^Q}}_0 = Q^{\pmb {\mu^{\hat Q}}}_0 =Q_0\). Then, it holds that

\begin{equation}\label{eq:33333}
\begin{bmatrix}
    \| Q^{\pmb \mu^{\pmb Q}}_{-1} - Q^{\pmb \mu^{\pmb {\hat Q}}}_{-1} \|_{\infty} \\
    \| Q^{\pmb \mu^{\pmb Q}}_{-2} - Q^{\pmb \mu^{\pmb {\hat Q}}}_{-2} \|_{\infty} \\
    \vdots \\
    \| Q^{\pmb \mu^{\pmb Q}}_{-T+1} - Q^{\pmb \mu^{\pmb {\hat Q}}}_{-T+1} \|_{\infty} \\
    \| Q^{\pmb \mu^{\pmb Q}}_{-T} - Q^{\pmb \mu^{\pmb {\hat Q}}}_{-T} \|_{\infty}
\end{bmatrix}
\leq
\underbrace{\begin{bmatrix}
     0 & \frac{\bar L K_1}{\rho} & \cdots & \frac{\bar L K_{\mathrm D}^{T-3}K_1}{\rho} & \frac{\bar L K_{\mathrm D}^{T-2}K_1}{\rho} \\
     \beta & 0 & \cdots & \frac{\bar L K_{\mathrm D}^{T-4}K_1}{\rho} & \frac{\bar L K_{\mathrm D}^{T-3}K_1}{\rho} \\
     \vdots & \vdots & \ddots & \vdots & \vdots \\
     0 & 0 & \cdots & 0 & \frac{\bar L K_1}{\rho} \\
     0 & 0 & \cdots & \beta & 0
\end{bmatrix}}_{\mathcal {\tilde B}_{T-1}}
\begin{bmatrix}
    \|Q_{-1} - \hat Q_{-1}\|_{\infty} \\
    \|Q_{-2} - \hat Q_{-2}\|_{\infty} \\
    \vdots \\
    \|Q_{-T+1} - \hat Q_{-T+1}\|_{\infty} \\
    \|Q_{-T} - \hat Q_{-T}\|_{\infty}
\end{bmatrix},
\end{equation}
where \(\bar L\) and \(K_{\mathrm D}\) are defined as in Convention \ref{conv:5}.
\end{lemma}
\begin{proof}
    The result directly follows from Lemmas \ref{lem:7} and \ref{lem:8}.
\end{proof}

\begin{lemma}\label{lem:aa}
    For all \(T\in \mathbb N\), the matrix \(\mathcal {\tilde B}_{T-1}\) is nonnegative and irreducible. In particular, it admits a positive nonnegative eigenvector that corresponds to its spectral radius \(\rho(\mathcal {\tilde B}_{T-1})\) by the Perron-Frobenius theorem. If \(\rho(\mathcal {\tilde B}_{T-1})<1\), then the iterations $$ \mathcal C_{\beta}^{T}\ni\pmb Q \to Q^{\pmb {\mu^Q}}\in \mathcal C_{\beta}^{T}$$ defined in Lemma \ref{lem:33333} converges to a fixed point under any norm over \( \mathbb R^{|X| \times |A| \times T}\).
\end{lemma}
\begin{proof}
    Define 
    $$ \mathcal C_{\beta}^{T}\times \mathcal C_{\beta}^{T}\ni (\pmb Q, \pmb{\hat Q}) \mapsto \left ( \|Q^{\pmb {\mu^Q}}_{-t}-Q^{\pmb {\mu^{\hat Q}}}_{-t}\|_{\infty}\right)_t =: T(\pmb Q,\pmb{\hat Q}).$$ 
    Let \(v_T\) be a positive eigenvector such that \(\mathcal{\tilde B}^{\top}_{T-1} v _T = \rho(\mathcal {\tilde B}_{T-1})v_T.\) Then, under the Euclidean inner product, it holds that 
    \[
    \langle T(\pmb Q,\pmb{\hat Q}),v_T \rangle \le \langle \mathcal{\tilde B}_{T-1}(\|Q_{-t}-\hat Q_{-t}\|_{\infty})_t,v_T\rangle = \rho(\mathcal {\tilde B}_{T-1})\langle (\|Q_{-t}-\hat Q_{-t}\|_{\infty})_t,v_T\rangle.
    \]
    Note that \( \left| \langle (\|Q_{-t}-\hat Q_{-t}\|_{\infty})_t, v_T \rangle \right|\) induces a norm over $\mathcal C^T_{\beta}$. Furthermore, under this norm, the map \(\pmb Q \to Q^{\pmb {\mu^Q}}\) is contractive. Since all the norms on \( \mathbb R^{|X| \times |A| \times T}\) are equivalent, \(T\) is eventually contractive under any norm on \( \mathbb R^{|X| \times |A| \times T}\).
\end{proof}

In the next lemma, we calculate the asymptote of \(\rho(\mathcal {\tilde B}_{T-1})\) as \(T\to \infty\).

\begin{lemma}\label{lem:BT_variational}
For $n\ge 1$, recursively set
\(
M_n:=\widetilde{\mathcal B}_{n-1}^{\top},
\)
Further set
\( u_1:=0,\)
and for \(k \ge 2\) set
\(
u_k:=\frac{\bar LK_1}{\rho} K_{\mathrm D}^{k-2}.
\)
Equivalently,
\[
\sum_{k\ge 1}u_k r^{k-1}
=
\sum_{k\ge 2}\frac{\bar LK_1}{\rho} K_{\mathrm D}^{k-2}r^{k-1}
=
\frac{\bar LK_1r}{\rho(1-K_{\mathrm D}r)}.
\]
For $r\in(0,K_{\mathrm D}^{-1})$ (with the convention $K_{\mathrm D}^{-1}=+\infty$ if $K_{\mathrm D}=0$), define
\[
F(r):=\frac{\beta}{r}+\sum_{k\ge 1}u_k r^{k-1}
=
\frac{\beta}{r}+\frac{\bar LK_1r}{\rho(1-K_{\mathrm D}r)}.
\]
Then limit
\(
\rho(\widetilde{\mathcal B}_{\infty})
:= \lim_{T\to\infty}\rho(\widetilde{\mathcal B}_T)
\)
exists and is finite, and moreover
\( \rho(\widetilde{\mathcal B}_{\infty}) = \min_{0<r<K_{\mathrm D}^{-1}}F(r). \)
\end{lemma}

\begin{proof}
Embed $M_n$ into the top-left corner of $M_{n+1}$. Since all entries are nonnegative, we have
\(M_n\le M_{n+1}\) componentwise,
and therefore by monotonicity of spectral radius for nonnegative matrices, for all \(n\) it holds that
\( \rho(\widetilde{\mathcal B}_{n-1})=\rho(M_n)\le \rho(M_{n+1})=\rho(\widetilde{\mathcal B}_{n}). \)
Hence $\bigl(\rho(\widetilde{\mathcal B}_{n-1})\bigr)_{n\ge 1}$ is monotone increasing, an thus
\( \lim_{n\to\infty}\rho(\widetilde{\mathcal B}_{n-1})\in(0,+\infty)\cup\{\infty\} \) exists.

We now show $\lim_{n\to\infty}\rho(\widetilde{\mathcal B}_{n-1})< \infty$. Fix any $r\in(0,K_{\mathrm D}^{-1})$ and define
\( x^{(n)}:=(1,r,r^2,\dots,r^{n-1})^\top\in\mathbb R^n_{++}.\)
For $1\le i<n$, the $i$-th row of $M_n$ has the form
\(
(u_i,u_{i-1},\dots,u_1,\beta,0,\dots,0);
\)
hence
\begin{align*}
\frac{(M_nx^{(n)})_i}{x^{(n)}_i}
&=
\sum_{k=1}^i u_k r^{k-1}+\frac{\beta}{r}
\le
\sum_{k\ge 1}u_k r^{k-1}+\frac{\beta}{r}
=
F(r).
\end{align*}
For the last row ($i=n$), there is no $\beta$-term, and thus
\[
\frac{(M_nx^{(n)})_n}{x^{(n)}_n}
=
\sum_{k=1}^n u_k r^{k-1}
\le
\sum_{k\ge 1}u_k r^{k-1}
\le F(r).
\]
Thus,
\( M_nx^{(n)}\le F(r)x^{(n)}. \)
By the Collatz--Wielandt bound,
\(
\rho(\widetilde{\mathcal B}_{n-1})=\rho(M_n)\le F(r)\) for every \(n.\)
Since this holds for every \(r\in(0,K_{\mathrm D}^{-1})\), we get
\(
\sup_n \rho(\widetilde{\mathcal B}_{n-1})\le \inf_{0<r<K_{\mathrm D}^{-1}}F(r)<\infty.
\)
Therefore
\(
\rho(\widetilde{\mathcal B}_{\infty})
:=
\lim_{n\to\infty}\rho(\widetilde{\mathcal B}_{n-1})
\in(0,\infty).
\)

We claim that $F$ has a unique minimizer on $(0,K_{\mathrm D}^{-1})$.
Indeed,
\[
F(r)=\frac{\beta}{r}+\frac{\bar LK_1r}{\rho(1-K_{\mathrm D}r)},\qquad
F'(r)=-\frac{\beta}{r^2}+\frac{\bar LK_1}{\rho(1-K_{\mathrm D}r)^2},
\text{ and }
F''(r)=\frac{2\beta}{r^3}+\frac{2\frac{\bar LK_1}{\rho}K_{\mathrm D}}{(1-K_{\mathrm D}r)^3}>0.
\]
Hence $F$ is strictly convex on $(0,K_{\mathrm D}^{-1})$.
Moreover,
\( \lim_{r\downarrow 0}F(r)=+\infty,\) and \( \lim_{r\nearrow K_{\mathrm D}^{-1}}F(r)=+\infty\).
Therefore, $F$ attains its minimum at a unique point
\( t_*\in(0,K_{\mathrm D}^{-1}). \)
Set
\( m_*:=F(t_*)=\min_{0<r<K_{\mathrm D}^{-1}}F(r). \)

For each $n$, let $v^{(n)}=(v^{(n)}_1,\dots,v^{(n)}_n)^\top>0$ be a Perron vector of $M_n$,
normalized by
\(
v^{(n)}_n=1,\)
and \(
M_n v^{(n)}=\rho(\widetilde{\mathcal B}_{n-1})v^{(n)}.
\)
Because of the structure of \(M_n\), the eigenvalues are characterized by the following equations:

\smallskip
\noindent\emph{Row \(i\) (\(1\le i<n\)):}
\begin{equation}\label{eq:BT_evi}
\sum_{k=1}^{i}u_k\,v^{(n)}_{\,i-k+1} + \beta\,v^{(n)}_{\,i+1}
=
\rho(\widetilde{\mathcal B}_{n-1})\,v^{(n)}_{\,i}.
\end{equation}

\noindent\emph{Row \(n\):}
\begin{equation}\label{eq:BT_evn}
\sum_{k=1}^{n}u_k\,v^{(n)}_{\,n-k+1}
=
\rho(\widetilde{\mathcal B}_{n-1})\,v^{(n)}_{\,n}
=
\rho(\widetilde{\mathcal B}_{n-1}).
\end{equation}
Define the ratios
\(
\mathfrak q_i^{(n)}:=\frac{v^{(n)}_i}{v^{(n)}_{i+1}},
\)
\(
i=1,\dots,n-1.
\)
Next, we establish a uniform lower bound.
From \eqref{eq:BT_evi}, we obtain
\(
\rho(\widetilde{\mathcal B}_{n-1})v^{(n)}_i\ge \beta\,v^{(n)}_{i+1},
\)
and thus
\begin{equation}\label{eq:BT_lower_ratio}
\mathfrak q_i^{(n)}=\frac{v^{(n)}_i}{v^{(n)}_{i+1}}
\ge \frac{\beta}{\rho(\widetilde{\mathcal B}_{n-1})}
\ge \frac{\beta}{m_*}
=
\frac{\beta}{m_*},
\qquad
1\le i\le n-1.
\end{equation}

\smallskip
\noindent\emph{Uniform upper bound.}
Since \(u_2=\frac{\bar L K_1}{\rho}>0\), for \(1\le i\le n-2\), using \eqref{eq:BT_evi} with \(i+1\) in place of \(i\),
\[
\rho(\widetilde{\mathcal B}_{n-1})v^{(n)}_{i+1}
=
\sum_{k=1}^{i+1}u_k v^{(n)}_{i+2-k}+\beta v^{(n)}_{i+2}
\ge u_2 v^{(n)}_i
=
\frac{\bar L K_1}{\rho}\,v^{(n)}_i.
\]
Hence
\begin{equation}\label{eq:BT_upper_ratio_mid}
\mathfrak q_i^{(n)}
=
\frac{v^{(n)}_i}{v^{(n)}_{i+1}}
\le \frac{\rho}{\bar L K_1}\,\rho(\widetilde{\mathcal B}_{n-1})
\le \frac{\rho m_*}{\overline L K_1},
\qquad
1\le i\le n-2.
\end{equation}
For \(i=n-1\), \eqref{eq:BT_evn} gives
\[
\rho(\widetilde{\mathcal B}_{n-1})
=
\sum_{k=1}^{n}u_k v^{(n)}_{n-k+1}
\ge u_2 v^{(n)}_{n-1}
=
\frac{\bar L K_1}{\rho} v^{(n)}_{n-1} \implies \mathfrak q_{n-1}^{(n)}
=
\frac{v^{(n)}_{n-1}}{v^{(n)}_n}
=
v^{(n)}_{n-1}
\le \frac{\rho}{\bar L K_1}\,\rho(\widetilde{\mathcal B}_{n-1})
\le \frac{\overline L K_1m_*}{\rho}.
\]
Combining with \eqref{eq:BT_lower_ratio}, we have for all \(n\ge 2\) and \(1\le i\le n-1\),
\begin{equation}\label{eq:BT_ratio_bounds}
\beta(m_*)^{-1}
\le
\mathfrak q_i^{(n)}
\le
m_*\left(\frac{\bar L K_1}{\rho}\right)^{-1}.
\end{equation}

Fix \(i\ge 1\). By \eqref{eq:BT_ratio_bounds}, the sequence
\((\mathfrak q_i^{(n)})_{n\ge i+1}\) is bounded in \((0,\infty)\). Thus, by the Bolzano--Weierstrass Theorem,
it has a convergent subsequence. Applying a standard diagonal argument,
there exists a subsequence \((n_\ell)\) such that for every fixed \(i\ge 1\),
\begin{equation}\label{eq:BT_q_limit}
\mathfrak q_i^{(n_\ell)}\longrightarrow \mathfrak q_i\in(0,\infty)
\qquad (\ell\to\infty).
\end{equation}
Since \(\rho(\widetilde{\mathcal B}_{n-1})\to \rho(\widetilde{\mathcal B}_\infty)\), along the same subsequence we also have
\begin{equation}\label{eq:BT_p_subseq}
\lim_{\ell \to \infty} \rho(\widetilde{\mathcal B}_{n_\ell-1}) =\rho(\widetilde{\mathcal B}_\infty).
\end{equation}
Now fix \(i\ge 1\). For all sufficiently large \(\ell\), we have \(n_\ell>i\), and dividing
\eqref{eq:BT_evi} (for \(n=n_\ell\)) by \(v^{(n_\ell)}_i\) yields
\[
\rho(\widetilde{\mathcal B}_{n_\ell-1})
=
\sum_{k=1}^i u_k\,\frac{v^{(n_\ell)}_{i-k+1}}{v^{(n_\ell)}_i}
+\frac{\beta}{\mathfrak q_i^{(n_\ell)}}
=
\sum_{k=1}^i u_k \prod_{j=i-k+1}^{i-1}\mathfrak q_j^{(n_\ell)}
+\frac{\beta}{\mathfrak q_i^{(n_\ell)}}.
\]
Letting \(\ell\to\infty\) and using \eqref{eq:BT_q_limit}--\eqref{eq:BT_p_subseq}, we obtain
\begin{equation}\label{eq:BT_limit_eq_fixed_i}
\rho(\widetilde{\mathcal B}_\infty)
=
\sum_{k=1}^i u_k \prod_{j=i-k+1}^{i-1}\mathfrak q_j
+\frac{\beta}{\mathfrak q_i}.
\end{equation}

Define
\(
\mathfrak q_{\min}:=\liminf_{i\to\infty}\mathfrak q_i.
\)
By \eqref{eq:BT_ratio_bounds}, we have
\(
\beta(m_*)^{-1}\le \mathfrak q_i\le m_*\left( \frac{\overline LK_1}{\rho}\right)^{-1}
\)
for all \(i\),
and thus in particular
\(
\mathfrak q_{\min}\ge \beta(m_*)^{-1}>0.
\)

Fix \(\varepsilon>0\) and \(K\in\mathbb N\). By the definition of \(\liminf\), there exists \(N\) such that
\(
\mathfrak q_j\ge \mathfrak q_{\min}-\varepsilon\)
for all \(j\ge N,\)
and there are infinitely many indices \(i\) such that
\(
\mathfrak q_i\le \mathfrak q_{\min}+\varepsilon.
\)
Choose one such \(i\ge N+K\). Applying \eqref{eq:BT_limit_eq_fixed_i} at this index \(i\), and using the estimate
\(\mathfrak q_i\le \mathfrak q_{\min}+\varepsilon\), we get
\begin{equation}\label{eq:BT_d1}
\rho(\widetilde{\mathcal B}_\infty)
\ge
\frac{\beta}{\mathfrak q_{\min}+\varepsilon}
+\sum_{k=1}^K u_k \prod_{j=i-k+1}^{i-1}\mathfrak q_j.
\end{equation}
For \(1\le k\le K\), all indices \(j\in[i-k+1,i-1]\) lie in \([N,\infty)\); hence
\[
\mathfrak q_j\ge \mathfrak q_{\min}-\varepsilon
\quad\Longrightarrow\quad
\prod_{j=i-k+1}^{i-1}\mathfrak q_j
\ge
(\mathfrak q_{\min}-\varepsilon)^{k-1}.
\]
Substituting into \eqref{eq:BT_d1} yields
\[
\rho(\widetilde{\mathcal B}_\infty)
\ge
\frac{\beta}{\mathfrak q_{\min}+\varepsilon}
+\sum_{k=1}^K u_k(\mathfrak q_{\min}-\varepsilon)^{k-1}.
\]
Since this is true for every \(K\), we let \(K\to\infty\) and obtain
\begin{equation}\label{eq:BT_d2}
\rho(\widetilde{\mathcal B}_\infty)
\ge
\frac{\beta}{\mathfrak q_{\min}+\varepsilon}
+\sum_{k\ge 1}u_k(\mathfrak q_{\min}-\varepsilon)^{k-1}.
\end{equation}

We already know \(\mathfrak q_{\min}>0\). We now show that
\(
\mathfrak q_{\min}<K_{\mathrm D}^{-1}
\)
(with the convention that \(K_{\mathrm D}^{-1}=+\infty\) if \(\frac{\overline LK_1}{\rho}=0).\)
If \(\frac{\overline LK_1}{\rho}=0\), this consequence is vacuous. Thus, it remains to consider the case \(\frac{\overline LK_1}{\rho}>0\). For the sake of contraction, if \(\mathfrak q_{\min}\ge K_{\mathrm D}^{-1}\), then for sufficiently small \(\varepsilon>0\),
\(
\mathfrak q_{\min}-\varepsilon\ge K_{\mathrm D}^{-1}.
\)
But then
\[
\sum_{k\ge 1}u_k(\mathfrak q_{\min}-\varepsilon)^{k-1}
=
\sum_{k\ge 2} \frac{\overline LK_1}{\rho} {K_{\mathrm D}}^{k-2}(\mathfrak q_{\min}-\varepsilon)^{k-1}
=
\frac{\overline LK_1}{\rho}(\mathfrak q_{\min}-\varepsilon)\sum_{m\ge 0}\bigl(K_{\mathrm D}(\mathfrak q_{\min}-\varepsilon)\bigr)^m
\]
diverges (because \(K_{\mathrm D}(\mathfrak q_{\min}-\varepsilon)\ge 1\)), which contradicts \eqref{eq:BT_d2} together with
\(\rho(\widetilde{\mathcal B}_\infty)<\infty\). Therefore \(\mathfrak q_{\min}<K_{\mathrm D}^{-1}\).

Now let \(\varepsilon\downarrow 0\) in \eqref{eq:BT_d2}; by continuity of \(F\) on \((0,K_{\mathrm D}^{-1})\),
\begin{equation}\label{eq:BT_lower_liminf}
\rho(\widetilde{\mathcal B}_\infty)
\ge
\frac{\beta}{\mathfrak q_{\min}}+\sum_{k\ge 1}u_k \mathfrak q_{\min}^{k-1}
=
F(\mathfrak q_{\min}).
\end{equation}

Since \(\mathfrak q_{\min}\in(0,K_{\mathrm D}^{-1})\), by minimality of \(F\) under \(t_*\) we have
\(
m_*=\min_{0<r<K_{\mathrm D}^{-1}}F(r)\le F(\mathfrak q_{\min}).
\)
In particular, it holds that
\(
m_* \le F(\mathfrak q_{\min}) \le \rho(\widetilde{\mathcal B}_\infty) \le m_*.
\)
Therefore,
\(
\rho(\widetilde{\mathcal B}_\infty)=m_*=\min_{0<r<K_{\mathrm D}^{-1}}F(r).
\)
This completes the proof.
\end{proof}

Set
\(
M_n:=\widetilde{\mathcal B}_{n-1}^{\top}\in\mathbb R^{n\times n},
\)
where \(n=T+1\).
Recall that the limiting Perron root is characterized by
\[
\rho(\widetilde{\mathcal B}_\infty)
=
\min_{0<r<K_{\mathrm D}^{-1}}F(r),
\qquad
F(r):=\frac{\beta}{r}+\frac{\overline LK_1r}{\rho(1-K_{\mathrm D} r)}.
\]
Let \(t_*\in(0,K_{\mathrm D}^{-1})\) be the unique minimizer of \(F\), and define
\[
\bar t_*^{(n)}
:=
(t_*^{\,n-1},t_*^{\,n-2},\dots,1)^\top
\in
\mathbb R^n.
\]

We now use \(\bar t_*^{(n)}\) to construct a Lyapunov function that controls the convergence to the infinite-horizon problem. First, we apply Collatz--Wielandt-type bounds to obtain an upper bound on \(\rho(\widetilde{\mathcal B}_T)\) using geometric test functions, as in the proof of Lemma~\ref{lem:BT_variational}.

\begin{lemma}\label{lem:BT_weighted_super}
For every $T\in\mathbb N$, letting $n:=T+1$, we have
\[
\max_{1\le i\le n}
\frac{\bigl(\widetilde{\mathcal B}_{T}^{\top}\,\bar t_*^{(n)}\bigr)_i}{\bar t^{(n)}_{*,i}}
\le
F(t_*)
=
\rho(\widetilde{\mathcal B}_\infty).
\]
Equivalently, componentwise it holds that
\(
\widetilde{\mathcal B}_{T}^{\top}\,\bar t_*^{(n)}
\le
\rho(\widetilde{\mathcal B}_\infty)\,\bar t_*^{(n)}.
\)
\end{lemma}

\begin{proof}
We prove the bound row by row.

Write $\bar t:=\bar t_*^{(n)}$ to simplify notation, and
\(
\bar t_i=t_*^{\,n-i},\) for \(i=1,\dots,n.\)
Hence for any admissible set of indices,
\(
\bar t_{i+1}t_*=\bar t_i,
\)
\(
\bar t_{i+\ell}=t_*^{-\ell}\bar t_i,
\)
and
\(
\bar t_{i-\ell}=t_*^{\ell}\bar t_i.
\)

\smallskip
\noindent\emph{Case 1: first row ($i=1$).}
Since the first diagonal entry of $\widetilde{\mathcal B}_T^\top$ is $0$, the first row is
$\left (0, \beta , 0, 0,\dots,0 \right).$
Therefore
\(
\frac{(\widetilde{\mathcal B}_T^{\top}\bar t)_1}{\bar t_1}
\le
F(t_*)
\)
follows directly.

\smallskip
\noindent\emph{Case 2: rows $2\le i\le n-1$.}
Under \(M_n=\widetilde{\mathcal B}_T^\top\), where the vector $\bar t$ satisfies
the row identity used in the proof of Lemma~\ref{lem:BT_variational}:
\[
\frac{(M_n\bar t)_i}{\bar t_i}
=
\sum_{k=1}^i u_k t_*^{k-1}+\frac{\beta}{t_*}
\le
\sum_{k\ge 1}u_k t_*^{k-1}+\frac{\beta}{t_*}
=
F(t_*)
\qquad(1\le i\le n-1).
\]
Since \(M_n=\widetilde{\mathcal B}_T^\top\), this is exactly the estimate needed for the
left-weighted inequality in Lemma~\ref{lem:BT_weighted_inner} below.

So, for the purposes of the Lyapunov argument, the relevant componentwise estimate is
\( M_n\bar t\le F(t_*)\bar t.\)

\smallskip
\noindent\emph{Case 3: last row ($i=n$).}
Again using the row structure of \(M_n\),
\[
\frac{(M_n\bar t)_n}{\bar t_n}
=
\sum_{k=1}^n u_k t_*^{k-1}
\le
\sum_{k\ge 1}u_k t_*^{k-1}
\le
F(t_*).
\]
Combining the previous two bounds (for rows \(1,\dots,n-1\) and row \(n\)) gives
\(
M_n\bar t\le F(t_*)\bar t.
\)
Since \(F(t_*)=\rho(\widetilde{\mathcal B}_\infty)\) by Lemma~\ref{lem:BT_variational},
the proof is complete.
\end{proof}

Next, using the fact that we can use positive geometric test vectors, we construct our Lyapunov functions.

\begin{lemma}
\label{lem:BT_weighted_inner}
With the same notation as above, for every $T\in\mathbb N$ and every
positive vector $y\in\mathbb R^{T+1}_{+}$, we have
\[
\bigl\langle \bar t_*^{(T)},\, \widetilde{\mathcal B}_T y \bigr\rangle
\le
\rho(\widetilde{\mathcal B}_\infty)\,
\bigl\langle \bar t_*^{(T)},\, y \bigr\rangle
\qquad
\bigl(\bar t_*^{(T)}=(t_*^T,\dots,1)\bigr).
\]
\end{lemma}

\begin{proof}
By Lemma~\ref{lem:BT_weighted_super} (in the transpose formulation), we have
\(
\widetilde{\mathcal B}_T^{\top}\bar t_*^{(n)}
\le
\rho(\widetilde{\mathcal B}_\infty)\,\bar t_*^{(n)}
\)
componentwise. 
Now let \(y\in\mathbb R^n_+\). Since \(y\ge 0\), taking the Euclidean inner product with \(y\)
preserves the inequality:
\(
\langle \widetilde{\mathcal B}_T^{\top}\bar t_*^{(T+1)},\, y\rangle
\le
\rho(\widetilde{\mathcal B}_\infty)\,\langle \bar t_*^{(T+1)},\, y\rangle.
\)
In particular, the desired left-weight contraction form is
\(
\langle \bar t_*^{(T+1)},\, \widetilde{\mathcal B}_{T} y\rangle
\le
\rho(\widetilde{\mathcal B}_\infty)\,\langle \bar t_*^{(T+1)},\, y\rangle,
\)
which is exactly the claimed inequality.
\end{proof}

\begin{remark}
The weights we use in the lemma above are different from those used in Lemma \ref{lem:aa}.
\end{remark}

Using the Lyapunov function constructed above, we now prove Theorem \ref{thrm:b} by approximating infinite-horizon TRMFEs via finite-horizon TRMFEs. Furthermore, we derive explicit convergence rates.

\begin{theorem}\label{thrm:conv}
    Let $(\pmb Q^T,\pmb\mu^T) \in \mathrm{MFE}$ and \((\pmb Q^{\infty},\pmb \mu^{\infty}) \in \mathrm{TRMFE}\). Suppose that there exists $t_o>1$ such that $F(t_o)<1$. Then, for all $t \ll T,$ we have
    \[
    \|Q^{\infty}_{-t}|_T-Q^T_{-t}\|_{\infty} \le O\left( \frac 1{t^{T-t}_o}\right).
    \]
    In particular, $|\mathrm{TRMFE}|=1$, i.e., there exists a unique TRMFE under any given terminal \(Q\)-function in \(\mathcal C_{\beta}\).
\end{theorem}
\begin{proof}[Proof of Theorem \ref{thrm:conv}]
If $t_o \not = t_*$, then we can perturb $\mathcal {\tilde B}_T$ so that its spectral radius is $F(t_o)$. Thus, without loss of generality, we will suppose that $t_*>1$ for the remainder of this proof.

Let $(\pmb Q^T,\pmb \mu^T)=\left((Q^T_{-t})_t,(\mu^T_{-t})_{t=0}^{T-1}\right)$ be a finite-horizon MFE obtained under the horizon length starting from $\mu_0$ under the terminal \(Q\)-function \(Q_{0} \in \mathcal C_{\beta}\). Let \(\pmb Q^{\infty},\pmb \mu^{\infty}\) be the infinite-horizon TRMFE obtained under \(Q_0\). Then, under the Euclidean inner product $\langle \cdot, \cdot \rangle$, arguing as in Lemma \ref{lem:4}, and using Lemma \ref{lem:BT_weighted_super}, we obtain that
\begin{align*}
&\langle \hat t_*, \left(\|Q^{\infty}_{-t}|_T-Q^T_{-t}\|_{\infty}\right)_{t=1}^{T-1} \rangle \\&\le
\langle \hat t_*, {\mathcal{\tilde B_T}} \left(\| Q^{\infty}_{-t}|_T - Q^T_{-t}\|_{\infty}\right)_{t=1}^{T-1}\rangle +\langle \hat t_*, \bar{\mathcal B}\left( \underbrace{\|\mu^{\infty}_{-T}\|_{1},\|\mu^{\infty}_{-T}\|_{1},\cdots, \|\mu^{\infty}_{-T}\|_{1}}_{T-1 \text{ times }}\right)\rangle
\\&\le \rho(\mathcal B_{\infty}) \langle \hat t_*, \left(\| Q^{\infty}_{-t}|_T - Q^T_{-t}\|_{\infty}\right)_{t=1}^{T-1}\rangle 
 + \langle \hat t_*, \bar{\mathcal B}\left( \|\mu^{\infty}_{-T}\|_{1},\|\mu^{\infty}_{-T}\|_{1},\cdots, \|\mu^{\infty}_{-T}\|_{1}\right)\rangle,
\end{align*}
where $
    \bar{\mathcal B} 
    = 
    \begin{bmatrix}
    K_{\mathrm D}^T & 0 & \cdots & 0\\
    0 & K_{\mathrm D}^{T-1} & \cdots & 0\\
    \vdots & \vdots & \ddots & \vdots \\
    0 & 0 & \cdots & K_{\mathrm D} 
\end{bmatrix}.$
Here, the terms $\|\mu^{\infty}_{-T}\|_{\infty}$ are due to the fact that $\pmb \mu^{\infty}$ induces an infinite-horizon TRMFE, which can be directly obtained from Lemma \ref{lem:1} through recursive iterations as we did in the proof of Lemma \ref{lem:4}.

Since $\rho(\mathcal {\tilde B}_{\infty})<1$, we can rearrange the inequality above to obtain that 
\begin{align*}
    t^{T-t}_*\|Q^{\infty}_t|_T-Q^T_t\|_\infty &\le\langle \hat t_* , \left(\|Q^{\infty}_t|_T-Q^T_t\|_\infty\right)_{t=1}^{T-1} \rangle 
    \\&\le \frac 1{1-\rho(\mathcal {\tilde B}_{\infty})}\underbrace{\langle \hat t_*, \bar{\mathcal B}\left( \|\mu^{\infty}_T\|_{1},\|\mu^{\infty}_T\|_{1},\cdots, \|\mu^{\infty}_T\|_{1}\right)\rangle}_{=I}.
\end{align*}
It remains to control the term $I$. Note that, as $t_* K_{\mathrm D} < 1$, it holds that
\begin{align*}
I & \le  \max_i\|\mu^{\infty}_{-T}\|_{1} K_{\mathrm D} \sum_{j=1}^{T-1} K_{\mathrm D}^{T-j}t^{T-j}_* \le \max_i\|\mu^{\infty}_{-T}\|_{1}K_{\mathrm D}^2 t_* \frac{K_{\mathrm D}^{T-1} t_*^{T-1}-1}{K_{\mathrm D} t_*-1}.
\end{align*}
Consequently, for the differences at section $t$ throughout all the population groups, we obtain the inequality
\begin{align*}
    \|Q^{\infty}_{-t}|_T-Q^T_{-t}\|_\infty \le  \max_i\|\mu^{\infty}_T\|_{1}K_{\mathrm D}^2 t^{t+1}_* \frac{K_{\mathrm D}^{T-1}t_*-\frac1{t_*^{T-t}}}{K_{\mathrm D} t_* -1}.
\end{align*}
Since $K_{\mathrm D}<1$ and $t_* >1$, as $T \to \infty$, we obtain that $ \|Q^{\infty}_{-t}|_T-Q^T_{-t}\|_{\infty} \to 0$. Furthermore, since $t_*K_{\mathrm D} < 1$, the dominating term in the sum above is $\frac{1}{t^{T-t}_*}$, which gives us the desired error bound.

It remains to prove the uniqueness of the infinite-horizon non-stationary TRMFE under any given family of terminal \(Q\)-functions. For this purpose, we will demonstrate this for the terminal \(Q\)-function $Q_0$, without any loss of generality. For each $T$, there exists a unique finite-horizon TRMFE under \(Q_0\) as $\rho(\tilde{\mathcal B}_{\infty})<1$. Furthermore, the error bound above shows that state-measures obtained under this MFE converges to any infinite-horizon TRMFE. From the uniqueness of the limit, it follows that all the infinite-horizon non-stationary TRMFE are obtained under a single state-measure flow $\pmb \mu^{\infty}$. Lastly, one can use dynamical programming to induce a unique family of $Q$-functions under $\pmb \mu^{\infty}$. Due to the existence of the regularizer term in our setting, there is a unique policy induced by these $Q$-functions, which implies that there is a unique infinite-horizon TRMFE.
\end{proof}

Since the uniqueness criterion above relies on the value \(t_*\) that minimizes \(F\), we now provide a general result characterizing the existence of a point \(t_*>1\) such that \(F(t_*)<1\). Since the asymptotic behavior of \(\rho(\widetilde{\mathcal B}_T)\) agrees with \(\min_r F(r)\), this characterization allows us to describe the worst-case behavior of the system.

\begin{lemma}\label{lem:tstar-threshold}
Let
\[
F(r):=\frac{\beta}{r}+\frac{\bar L K_1r}{\rho(1-K_{\mathrm D}r)},\qquad r\in(0,K_{\mathrm D}^{-1}),
\]
where $\beta>0$, $\frac{\bar L K_1}{\rho}>0$, and $0\le K_{\mathrm D}<1$. Let $t_*$ be the unique minimizer of $F$ on
$(0,K_{\mathrm D}^{-1})$. Then:

\begin{enumerate}
    \item The minimizer is given by
    \[
    t_*=\frac{\sqrt{\beta}}{\sqrt \frac{\bar L K_1}{\rho}+K_{\mathrm D}\sqrt\beta}.
    \]
    \item We have
    \(
    t_*>1
    \) if, and only if,
    \(
    \frac{\bar L K_1}{\rho}<\beta(1-K_{\mathrm D})^2.
    \)
    Equivalently,
    \(
    \frac{\bar L K_1}{\rho}>\beta(1-K_{\mathrm D})^2
    \) if, and only if, \(
    t_*<1.
    \)
    \item The minimum value is
    \(
    F(t_*)=K_{\mathrm D}\beta+2\sqrt{\frac{\bar L K_1}{\rho}\beta},
    \)
    and therefore
    \(
    F(t_*)<1
    \) if, and only if, \(
    \frac{\bar L K_1}{\rho}<\frac{(1-K_{\mathrm D}\beta)^2}{4\beta}.
    \)
\end{enumerate}
\end{lemma}

\begin{proof}
We compute
\[
F'(r)=-\frac{\beta}{r^2}+\frac{\bar L K_1}{\rho(1-K_{\mathrm D}r)^2},
\qquad
F''(r)=\frac{2\beta}{r^3}+\frac{2\frac{\bar L K_1}{\rho}K_{\mathrm D}}{(1-K_{\mathrm D}r)^3}>0
\quad\text{for }r\in(0,K_{\mathrm D}^{-1}).
\]
Thus $F$ is strictly convex on $(0,K_{\mathrm D}^{-1})$.

Moreover,
\(
\lim_{r\downarrow 0}F(r)=+\infty\), and 
\(
\lim_{r\nearrow K_{\mathrm D}^{-1}}F(r)=+\infty,
\)
and thus $F$ attains a unique minimum at some $t_*\in(0,K_{\mathrm D}^{-1})$.
Since $t_*$ is the unique minimizer, it satisfies $F'(t_*)=0$, i.e.
\[
-\frac{\beta}{t_*^2}+\frac{\bar L K_1}{\rho(1-K_{\mathrm D} t_*)^2}=0.
\]
After rearranging, we obtain
\begin{equation}\label{eq:exp}
t_*=\frac{\sqrt\beta}{\sqrt \frac{\bar L K_1}{\rho}+K_{\mathrm D}\sqrt\beta}.
\end{equation}
Notice that \eqref{eq:exp} implies that
\(
t_*>1\) if, and only if, \(
\frac{\bar L K_1}{\rho}<\beta(1-K_{\mathrm D})^2,\)
and that \(
\frac{\bar L K_1}{\rho}>\beta(1-K_{\mathrm D})^2\) if, and only if, \(t_*<1\)
follows immediately.
Using the identity from the critical point equation,
\(
\frac{\sqrt\beta}{t_*}=\frac{\sqrt \frac{\bar L K_1}{\rho}}{1-K_{\mathrm D} t_*},
\)
we get
\[
t_*=\frac{\sqrt\beta}{\sqrt \frac{\bar L K_1}{\rho}+ K_{\mathrm D}\sqrt\beta}.
\]
Therefore,
\[
F(t_*)
=
\frac{\beta}{t_*}+\frac{\frac{\bar L K_1}{\rho} t_*}{1-K_{\mathrm D} t_*}
=
(\sqrt{\frac{\bar L K_1}{\rho}\beta}+K_{\mathrm D}\beta)+\sqrt{\frac{\bar L K_1}{\rho}\beta}
=
K_{\mathrm D}\beta+2\sqrt{\frac{\bar L K_1}{\rho}\beta}.
\]

From the previous step,
\[
F(t_*)<1
\quad\Longleftrightarrow\quad
K_{\mathrm D}\beta+2\sqrt{\frac{\bar L K_1}{\rho}\beta}<1
\quad\Longleftrightarrow\quad
2\sqrt{\frac{\bar L K_1}{\rho}\beta}<1-K_{\mathrm D}\beta.
\]
Therefore,
\[
t_*>1\ \text{and}\ F(t_*)<1
\quad\Longleftrightarrow\quad
\frac{\bar L K_1}{\rho}<\min\!\left\{\beta(1-K_{\mathrm D})^2,\ \frac{(1-K_{\mathrm D}\beta)^2}{4\beta}\right\}.
\]
This completes the proof.
\end{proof}

\begin{proof}[Proof of Theorem \ref{thrm:b}]
    The theorem follows directly from Theorem \ref{thrm:conv} and Lemma \ref{lem:tstar-threshold}.
\end{proof}

\section{One-Shot Terminal Value Invariance for the Mean-Field Terminal Value Problem under Fixed Initial State-Measure}\label{sect:one-shot}

In this section, we prove Theorem \ref{thrm:3}.

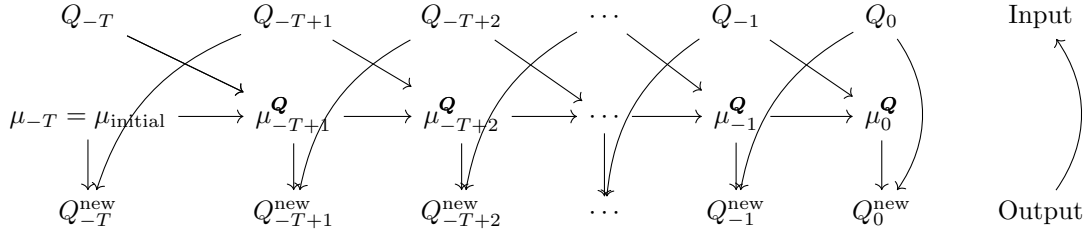
\begin{figure}[H]
\begin{tikzcd}
Q_{-T} \arrow[dr] \arrow[dr] & Q_{-T+1} \arrow[dr] \arrow[ddl, bend right=25] &Q_{-T+2} \arrow[dr] \arrow[ddl, bend right=25] & \cdots \arrow[dr] \arrow[ddl, bend right=25] & Q_{-1} \arrow[dr] \arrow[ddl, bend right=25] & Q_{0} \arrow[ddl, bend right=25]\arrow[dd, bend left=35]& \mathrm{Input} \\
\mu_{-T}=\mu_{\mathrm{initial}} \arrow[r] \arrow[d] & \mu^{\pmb Q}_{-T+1} \arrow[r]\arrow[d] & \mu^{\pmb Q}_{-T+2} \arrow[r] \arrow[d] & \cdots \arrow[r] \arrow[d] & \mu^{\pmb Q}_{-1} \arrow[r] \arrow[d] & \mu^{\pmb Q}_0 \arrow[d] &\\
Q^{\mathrm{new}}_{-T} & Q^{\mathrm{new}}_{-T+1} & Q^{\mathrm{new}}_{-T+2} & \cdots & Q^{\mathrm{new}}_{-1} & Q^{\mathrm{new}}_0  &\mathrm{Output} \arrow[uu, bend right=35]
\end{tikzcd}
\caption{Depiction of one-step invariance forced on the terminal $Q$-functions throughout iterations in the finite-horizon case}
\label{fig:2}
\end{figure}

Below, in an algorithm environment, we provide a summary of the iteration process we use for the proof of Theorem \ref{thrm:3}.

\begin{algorithm}[H]\label{alg:1}
\caption{Value Iteration with Terminal Invariance}
    Fix initial state-measure $\mu_0 \in \mathcal P(X)$\;
    Initialize with $\pmb Q^T_0 = (Q_{0,t})_{t=0}^T \in \prod_{t=0}^T \mathcal C$\;
    \While{$\pmb Q^{T}_{n+1} \not = \pmb Q^{T}_n$}{
        \For{$t=0,\cdots,T,T+1$}{
            $\mu_{t+1}= H_2(Q_{n,t},\mu_t)$
        }
         \For{$t=T-1,T-2,\cdots,0$}{
            $Q_{n+1,t} = H_1(Q_{n,t+1},\mu_{t})$\;
        }
        $Q_{n+1,T} = H_1(Q_T,\mu_{T+1})$\;
        $\pmb Q^T_{n+1} = (Q_{n+1,t})_{t=0}^T$\;
    }
    Find an accumulation point $(\pmb \pi^*,(\mu^{*,0},\pmb \mu^*))$ obtained from the disintegration of weak-limit accumulation point of $\pmb \pi^{*,T}\otimes \pmb \mu^{*,T}$ as $T \to \infty$\;
\Return{stationary mean-field equilibrium $(\pmb \pi^*,\pmb \mu^*)$}
\end{algorithm}

Let $\mu_0 \in \mathcal P(X)$ be a fixed initial state-measure. The scheme summarized in Algorithm \ref{alg:1} depicted in Figure \ref{fig:2}, where the iterations satisfy the following vector inequality:
\begin{equation}\label{eq:123455}
\begin{bmatrix}
    \| Q^{\pmb \mu^{\pmb Q}}_0 - Q^{\pmb \mu^{\pmb {\hat Q}}}_0 \|_{\infty} \\
    \| Q^{\pmb \mu^{\pmb Q}}_{-1} - Q^{\pmb \mu^{\pmb {\hat Q}}}_{-1} \|_{\infty} \\
    \| Q^{\pmb \mu^{\pmb Q}}_{-2} - Q^{\pmb \mu^{\pmb {\hat Q}}}_{-2} \|_{\infty} \\
    \vdots \\
    \| Q^{\pmb \mu^{\pmb Q}}_{-T+1} - Q^{\pmb \mu^{\pmb {\hat Q}}}_{-T+1} \|_{\infty} \\
    \| Q^{\pmb \mu^{\pmb Q}}_{-T} - Q^{\pmb \mu^{\pmb {\hat Q}}}_{-T} \|_{\infty}
\end{bmatrix}
\leq
\underbrace{\begin{bmatrix}
    \beta & \frac{\bar L K_1}{\rho} & \frac{\bar L K_{\mathrm D} K_1}{\rho} & \cdots & \frac{\bar L K_{\mathrm D}^{T-2}K_1}{\rho} & \frac{\bar L K_{\mathrm D}^{T-1}K_1}{\rho} \\
    \beta & 0 & \frac{\bar L K_1}{\rho} & \cdots & \frac{\bar L K_{\mathrm D}^{T-3}K_1}{\rho} & \frac{\bar L K_{\mathrm D}^{T-2}K_1}{\rho} \\
    0 & \beta & 0 & \cdots & \frac{\bar L K_{\mathrm D}^{T-4}K_1}{\rho} & \frac{\bar L K_{\mathrm D}^{T-3}K_1}{\rho} \\
    \vdots & \vdots & \vdots & \ddots & \vdots & \vdots \\
    0 & 0 & 0 & \cdots & 0 & \frac{\bar L K_1}{\rho} \\
    0 & 0 & 0 & \cdots & \beta & 0
\end{bmatrix}}_{\mathcal B_T}
\begin{bmatrix}
    \|Q_0 - \hat Q_0\|_{\infty} \\
    \|Q_{-1} - \hat Q_{-1}\|_{\infty} \\
    \|Q_{-2} - \hat Q_{-2}\|_{\infty} \\
    \vdots \\
    \|Q_{-T+1} - \hat Q_{-T+1}\|_{\infty} \\
    \|Q_{-T} - \hat Q_{-T}\|_{\infty}
\end{bmatrix},
\end{equation}
where \(Q_t^{\pmb {\mu^Q}}\) are defined as in Lemma \ref{lem:33333}, with the exception that the terminal \(Q\)-functions vary. The appearance of the term \(\beta\) on the first row and column on \(\mathcal B_T\) is due to the fact that the terminal \(Q\)-function varies.

It turns out that, unlike $\tilde{\mathcal B}_T$, the spectral radius of $\mathcal B_T$ has two different asymptotics separated by a phase-shift condition, see Lemma \ref{lem:BT_variational}. In this section, our main result will be an identication of the exact phase-shift required for $\mathcal B_T$ and the limit of the spectral radius in each case.

\newcommand{\rB}{\rho(\mathcal B_T)}
We will write $\rho(\mathcal B_T):=\rho(\mathcal B_T)$.

\begin{lemma}\label{lem:BT_basic}
Assume $\frac{\bar L K_1}{\rho}>0$, $\beta>0$, and $K_{\mathrm D}\geq0$. Then, the following hold.

\begin{enumerate}
    \item $\mathcal B_T$ is irreducible and has nonnegative entries. Hence, by the Perron--Frobenius Theorem, $\rho(\mathcal B_T)>0$ is an eigenvalue of $\mathcal B_T$.
    \item The sequence $(\rho(\mathcal B_T))_{T\in\mathbb N}$ is strictly increasing.
    \item If $K_{\mathrm D}<1$, then, for every $T\in\mathbb N$,
    \[
    \beta<\rho(\mathcal B_T)
    \le \|\mathcal B_T\|_\infty
    =\beta+\frac{\bar LK_1}{\rho}\sum_{j=0}^{T-1}K_{\mathrm D}^j
    \le \beta+\frac{\bar LK_1}{\rho(1-K_{\mathrm D})}.
    \]
\end{enumerate}
\end{lemma}

\begin{proof}
All entries of $\mathcal B_T$ are clearly nonnegative. It is also straightforward to see that $\mathcal B_T$ is irreducible. The Perron--Frobenius Theorem now implies that $\rho(\mathcal B_T)>0$ is an eigenvalue of $\mathcal B_T$.

The matrix $\mathcal B_T$ is a proper leading principal submatrix of $\mathcal B_{T+1}$, and both are nonnegative. Since $\mathcal B_{T+1}$ is irreducible, the Perron root strictly increases when passing from a proper principal submatrix to the full matrix. Hence
\(
\rho(\mathcal B_T)<\rho(\mathcal B_{T+1})\)
for all \(T\).

For the third assertion, suppose that $K_{\mathrm D}<1$.  The upper bound follows from the operator norm estimate
\(
\rho(\mathcal B_T)\le \|\mathcal B_T\|_\infty
=\max_{1\le i\le T+1}\sum_{j=1}^{T+1} (\mathcal B_T)_{ij}.
\)
The first row has the largest row sum, namely
\[
\beta+\frac{\bar LK_1}{\rho}+\frac{\bar LK_1}{\rho}K_{\mathrm D}+\cdots+\frac{\bar LK_1}{\rho}K_{\mathrm D}^{T-1}
=\beta+\frac{\bar LK_1}{\rho}\sum_{j=0}^{T-1}K_{\mathrm D}^j,
\]
which yields
\[
\rho(\mathcal B_T)\le \beta+\frac{\bar LK_1}{\rho}\sum_{j=0}^{T-1}K_{\mathrm D}^j\le \beta+\frac{\bar LK_1}{\rho(1-K_{\mathrm D})}.
\]

For the strict lower bound, note that the vector $e_1=(1,0,\dots,0)^\top$ satisfies
\(
\mathcal B_T e_1=\beta e_1+\beta e_2,
\)
and thus $\mathcal B_T$ strictly dominates the scalar action $\beta I$ on a nonzero direction. Since $\mathcal B_T$ is irreducible and not equal to $\beta I$, the Perron--Frobenius Theorem implies $\rho(\mathcal B_T)>\beta$.
\end{proof}
\begin{lemma}\label{lem:BT_recurrence}
For \(n\ge 1\), consider the characteristic determinant
\(\det(\lambda I_n-\mathcal B_{n-1})\), with the zeroth-order determinant taken to be \(1\). Then,
\(
\det(\lambda I_1-\mathcal B_0)=\lambda-\beta,
\) and
\(
\det(\lambda I_2-\mathcal B_1)=\lambda(\lambda-\beta)-\frac{\bar L K_1}{\rho}\beta,
\)
and for all $n\ge 3$,
\begin{equation}\label{eq:BT_det_recurrence}
\det(\lambda I_n-\mathcal B_{n-1})
=
(\lambda+K_{\mathrm D}\beta)\,\det(\lambda I_{n-1}-\mathcal B_{n-2})
-\beta\left (\frac{\bar L K_1}{\rho}+K_{\mathrm D}\lambda\right )\,\det(\lambda I_{n-2}-\mathcal B_{n-3}).
\end{equation}
\end{lemma}

\begin{proof}
We derive a tridiagonal representation for the characteristic polynomial.

Let $U_n$ denote the $n\times n$ upper-shift matrix (i.e., $(U_n)_{i,i+1}=1$ and all other entries are zero), and also let $L_n:=U_n^\top$ be the lower-shift matrix, and let $E_{11}$ be the matrix with a $1$ in the $(1,1)$ entry and zeros elsewhere. Then, one checks directly from the definition of $\mathcal B_{n-1}$ that
\[
\mathcal B_{n-1}
=
\beta L_n+\beta E_{11}+\frac{\bar L K_1}{\rho}\,U_n(I_n-K_{\mathrm D} U_n)^{-1}.
\]
Indeed, $U_n(I_n-K_{\mathrm D}U_n)^{-1}=U_n+K_{\mathrm D}U_n^2+\cdots+{K_{\mathrm D}}^{n-1}U_n^n$ produces exactly the geometric superdiagonals in $\mathcal B_{n-1}$.

Therefore,
\[
\lambda I_n-\mathcal B_{n-1}
=
\lambda I_n-\beta L_n-\beta E_{11}-\frac{\bar L K_1}{\rho}\,U_n(I_n-K_{\mathrm D}U_n)^{-1}.
\]
Multiplying the right-hand side by \(I_n-K_{\mathrm D} U_n\) and using $\det(I_n-K_{\mathrm D} U_n)=1$ (it is upper triangular with unit diagonal), we obtain
\begin{align*}
\det(\lambda I_n-\mathcal B_{n-1})=\det \underbrace{\Big((\lambda I_n-\beta L_n-\beta E_{11})(I_n-K_{\mathrm D}U_n)-\frac{\bar L K_1}{\rho}U_n\Big)}_{=T_n(\lambda)}.
\end{align*}
A direct computation shows that $T_n(\lambda)$ is tridiagonal, with entries:
\begin{itemize}
    \item diagonal:
    \(
    d_1=\lambda-\beta,\) and \(d_i=\lambda+K_{\mathrm D}\beta\), where \(2\le i\le n,
    \)
    \item subdiagonal (all rows): \(\ell_i=-\beta\) for \(1\le i\le n-1\),
    \item superdiagonal:
    \[
    u_1=-\left(\frac{\bar L K_1}{\rho}+K_{\mathrm D}(\lambda-\beta)\right),
    \qquad
    u_i=-\left(\frac{\bar L K_1}{\rho}+K_{\mathrm D}\lambda\right)\quad (2\le i\le n-1).
    \]
\end{itemize}

Since \(\det(\lambda I_n-\mathcal B_{n-1})=\det T_n(\lambda)\), the standard continuant recurrence for tridiagonal determinants gives
\[
\det(\lambda I_n-\mathcal B_{n-1})
=d_n\det(\lambda I_{n-1}-\mathcal B_{n-2})
-\ell_{n-1}u_{n-1}\det(\lambda I_{n-2}-\mathcal B_{n-3}),
\qquad n\ge 3.
\]
Substituting the values above yields
\[
\det(\lambda I_n-\mathcal B_{n-1})
=
(\lambda+K_{\mathrm D}\beta)\det(\lambda I_{n-1}-\mathcal B_{n-2})
-\beta\left (\frac{\bar L K_1}{\rho}+K_{\mathrm D}\lambda \right)\det(\lambda I_{n-2}-\mathcal B_{n-3}),
\]
which is exactly \eqref{eq:BT_det_recurrence}, where the first- and second-order determinants are defined through the matrices
\[
\mathcal B_0=[\beta],
\qquad
\mathcal B_1=\begin{bmatrix}\beta&\frac{\bar L K_1}{\rho}\\ \beta&0\end{bmatrix}
\]
so that the recursion has a well-defined starting point.
\end{proof}


Lastly, we calculate the asymptotes of \(\rho(\mathcal B_T)\), from which Theorem \ref{thrm:3} follows after one uses \eqref{eq:123455} (see Lemma \ref{lem:aa}).
\begin{theorem}\label{thm:BT_asymptotics}
Equivalently, in the original constants,
\[
\rho(\mathcal B_T)\nearrow
\begin{cases}
\displaystyle
\beta+\frac{\bar L K_1}{\rho(1-K_{\mathrm D})},
& \displaystyle 0\leq K_{\mathrm D}<1
\text{ and }
\frac{\bar L K_1}{\rho}<\beta(1-K_{\mathrm D})^2,\\[3mm]
\displaystyle
\beta K_{\mathrm D}+2\sqrt{\beta\frac{\bar L K_1}{\rho}},
& \displaystyle K_{\mathrm D}\geq1
\text{ or }
\frac{\bar L K_1}{\rho}>\beta(1-K_{\mathrm D})^2.
\end{cases}
\]
At the boundary $0\leq K_{\mathrm D}<1$ and
$\frac{\bar L K_1}{\rho}=\beta(1-K_{\mathrm D})^2$, the two displayed values agree.
\end{theorem}

\begin{proof}
Let \(n=T+1\). By Lemma~\ref{lem:BT_recurrence}, the characteristic determinant
\(\det(\lambda I_n-\mathcal B_{n-1})\)
satisfies the following second-order recurrence for any \(n \ge 3\):
\begin{equation}\label{eq:poly1}
\det(\lambda I_n-\mathcal B_{n-1})
=(\lambda+K_{\mathrm D}\beta)\det(\lambda I_{n-1}-\mathcal B_{n-2})
-\beta\left(\frac{\bar L K_1}{\rho}+K_{\mathrm D}\lambda\right)
\det(\lambda I_{n-2}-\mathcal B_{n-3}).
\end{equation}
For fixed \(\lambda\), the associated characteristic equation to the recursion \eqref{eq:poly1} is
\begin{equation}\label{eq:BT_char_eq}
r^2-(\lambda+K_{\mathrm D}\beta)r+\beta(\frac{\bar L K_1}{\rho}+K_{\mathrm D}\lambda)=0.
\end{equation}
The discriminant that corresponds to \eqref{eq:BT_char_eq} is
\[
\Delta(\lambda):=(\lambda+K_{\mathrm D}\beta)^2-4\beta(\frac{\bar L K_1}{\rho}+K_{\mathrm D}\lambda)
=(\lambda-K_{\mathrm D}\beta)^2-4\frac{\bar L K_1}{\rho}\beta;
\]
and thus,
\(
\Delta(\lambda)=0
\) if, and only if, \(
\lambda=K_{\mathrm D}\beta\pm 2\sqrt{\frac{\bar L K_1}{\rho}\beta}.
\)
Since \(\lambda>0\) is the relevant regime for the Perron root, the critical value is
\(
\lambda_{\mathrm{ac}}=K_{\mathrm D}\beta+2\sqrt{\frac{\bar L K_1}{\rho}\beta}.
\)

Fix \(\lambda>\lambda_{\mathrm{ac}}\), so that \(\Delta(\lambda)>0\). By
\[
r_\pm(\lambda)
:=
\frac{\lambda+K_{\mathrm D}\beta\pm \sqrt{\Delta(\lambda)}}{2},
\qquad
r_+(\lambda)>r_-(\lambda)>0,
\]
denote the corresponding roots.
Then, for \(n \ge 1\), the recurrence \eqref{eq:BT_det_recurrence} yields
\begin{equation}\label{eq:BT_pn_realroots}
\det(\lambda I_n-\mathcal B_{n-1})
=c_+(\lambda)\,r_+(\lambda)^{n-1}+c_-(\lambda)\,r_-(\lambda)^{n-1},
\end{equation}
where the coefficients are determined by \(\det(\lambda I_1-\mathcal B_0)=\lambda-\beta\) and \(\det(\lambda I_2-\mathcal B_1)=\lambda(\lambda-\beta)-\frac{\bar L K_1}{\rho}\beta\), namely
\begin{equation}\label{eq:BT_cpm}
c_+(\lambda)=\frac{\lambda(\lambda-\beta)-\frac{\bar L K_1}{\rho}\beta-r_-(\lambda)(\lambda-\beta)}{r_+(\lambda)-r_-(\lambda)},
\qquad
c_-(\lambda)=\frac{r_+(\lambda)(\lambda-\beta)-\lambda(\lambda-\beta)+\frac{\bar L K_1}{\rho}\beta}{r_+(\lambda)-r_-(\lambda)}.
\end{equation}

We first treat the case $K_{\mathrm D}\geq1$, which is not covered by the
geometric-series estimate in Lemma~\ref{lem:BT_basic}.  We claim that in this
case
\begin{equation}\label{eq:BT-large-K-limit}
\rho(\mathcal B_T)\nearrow
\lambda_{\mathrm{ac}}
:=K_{\mathrm D}\beta+2\sqrt{\frac{\bar L K_1}{\rho}\beta}.
\end{equation}
For $\lambda>\lambda_{\mathrm{ac}}$, the coefficient $c_+(\lambda)$ in
\eqref{eq:BT_pn_realroots} can vanish only if
\begin{equation}\label{eq:BT-boundary-zero}
(1-K_{\mathrm D})(\lambda-\beta)=\frac{\bar L K_1}{\rho}.
\end{equation}
Indeed, this follows by substituting the expression for $r_-(\lambda)$ forced
by $c_+(\lambda)=0$ into the characteristic equation
\eqref{eq:BT_char_eq}.  If $K_{\mathrm D}=1$, equation
\eqref{eq:BT-boundary-zero} does not hold because
$\frac{\bar L K_1}{\rho}>0$.  If $K_{\mathrm D}>1$, its only solution is
\[
\lambda=\beta-\frac{\bar L K_1}{\rho(K_{\mathrm D}-1)}
<\beta<\lambda_{\mathrm{ac}}.
\]
Consequently, $c_+$ has no zero on
$(\lambda_{\mathrm{ac}},\infty)$.  Since \eqref{eq:BT_cpm} gives
$c_+(\lambda)>0$ for all sufficiently large $\lambda$, continuity yields
$c_+(\lambda)>0$ throughout $(\lambda_{\mathrm{ac}},\infty)$.  Thus, for
every fixed $\lambda>\lambda_{\mathrm{ac}}$, the dominant term in
\eqref{eq:BT_pn_realroots} is positive for all sufficiently large $n$.
In fact, for any fixed $\lambda_0>\lambda_{\mathrm{ac}}$, this dominance is
uniform for $\lambda\geq\lambda_0$: on compact $\lambda$-intervals this
follows from continuity and $r_-/r_+<1$, while for large $\lambda$ it follows
directly from \eqref{eq:BT_cpm} and $r_-/r_+\to0$.  Hence, for all sufficiently
large $n$, the characteristic polynomial is positive on
$[\lambda_0,\infty)$ and has no real root there.  In particular, its Perron
root is smaller than $\lambda_0$.  Therefore $\limsup_{T\to\infty}\rho(\mathcal B_T)\leq\lambda_{\mathrm{ac}}.$

For the reverse inequality, fix $\max\left\{0,K_{\mathrm D}\beta-2\sqrt{\frac{\bar L K_1}{\rho}\beta}\right\} <\lambda<\lambda_{\mathrm{ac}}. $
Then $\Delta(\lambda)<0$, and the two roots of
\eqref{eq:BT_char_eq} are complex conjugates.  Hence
\[
\det(\lambda I_n-\mathcal B_{n-1})
=m(\lambda)^{n-1}
\left(A(\lambda)\cos((n-1)\theta(\lambda))
+B(\lambda)\sin((n-1)\theta(\lambda))\right),
\]
where $m(\lambda)>0$, $\theta(\lambda)\in(0,\pi)$, and
$(A(\lambda),B(\lambda))\neq(0,0)$.  The trigonometric factor is negative
for infinitely many $n$.  For those $n$, the characteristic polynomial has
a real root larger than $\lambda$, and therefore
$\rho(\mathcal B_{n-1})>\lambda$.  Since the Perron roots increase with $T$
by Lemma~\ref{lem:BT_basic}, letting
$\lambda\nearrow\lambda_{\mathrm{ac}}$ gives
\[
\liminf_{T\to\infty}\rho(\mathcal B_T)\geq\lambda_{\mathrm{ac}}.
\]
This proves \eqref{eq:BT-large-K-limit}.  For the remainder of the proof, we
may therefore suppose that $0\leq K_{\mathrm D}<1$.

Define
\(
\lambda_{\mathrm{bd}}:=\beta+\frac{\bar L K_1}{\rho(1-K_{\mathrm D})}.
\)
We claim that \(\lambda_{\mathrm{ac}}\le \lambda_{\mathrm{bd}}\). Indeed,
\begin{align*}
\lambda_{\mathrm{bd}}-\lambda_{\mathrm{ac}}
=
\beta+\frac{\bar L K_1}{\rho(1-K_{\mathrm D})}-K_{\mathrm D}\beta-2\sqrt{\frac{\bar L K_1}{\rho}\beta}
&=
\beta(1-K_{\mathrm D})+\frac{\bar L K_1}{\rho(1-K_{\mathrm D})}-2\sqrt{\frac{\bar L K_1}{\rho}\beta}\\
&=
\left(\sqrt{\beta(1-K_{\mathrm D})}-\sqrt{\frac{\bar L K_1}{\rho(1-K_{\mathrm D})}}\right)^2
\ge 0.
\end{align*}
Clearly, equality holds if, and only if, \(\frac{\bar L K_1}{\rho}=\beta(1-K_{\mathrm D})^2\).

Notice that
\[
\frac{\lambda_{\mathrm{bd}}(\lambda_{\mathrm{bd}}-\beta)-\frac{\bar L K_1}{\rho}\beta}
{\lambda_{\mathrm{bd}}-\beta}
=
\lambda_{\mathrm{bd}}-\frac{\frac{\bar L K_1}{\rho}\beta}{\lambda_{\mathrm{bd}}-\beta}
=
\lambda_{\mathrm{bd}}-\beta(1-K_{\mathrm D})
=K_{\mathrm D}\beta+\frac{\bar L K_1}{\rho(1-K_{\mathrm D})}.
\]
Moreover, substituting \(\lambda=\lambda_{\mathrm{bd}}\) into \eqref{eq:BT_char_eq},
one checks directly that the two roots are precisely \(\beta\) and \(K_{\mathrm D}\beta+\frac{\bar L K_1}{\rho(1-K_{\mathrm D})}\):
\[
r^2-(\lambda_{\mathrm{bd}}+K_{\mathrm D}\beta)r+\beta(\frac{\bar L K_1}{\rho}+K_{\mathrm D}\lambda_{\mathrm{bd}})
=(r-\beta)\left(r-K_{\mathrm D}\beta-\frac{\bar L K_1}{\rho(1-K_{\mathrm D})}\right).
\]
Hence
\(
\{r_+(\lambda_{\mathrm{bd}}),\,r_-(\lambda_{\mathrm{bd}})\}
=\left\{\beta,K_{\mathrm D}\beta+\frac{\bar L K_1}{\rho(1-K_{\mathrm D})}\right\}.
\)

Now note:
\[
K_{\mathrm D}\beta+\frac{\bar L K_1}{\rho(1-K_{\mathrm D})}>\beta
\iff
K_{\mathrm D}\beta+\frac{\bar L K_1}{\rho(1-K_{\mathrm D})}>\beta
\iff
\frac{\bar L K_1}{\rho}>\beta(1-K_{\mathrm D})^2.
\]
Thus we have obtained that:
\begin{enumerate}
    \item if \(\frac{\bar L K_1}{\rho}<\beta(1-K_{\mathrm D})^2\), then \(K_{\mathrm D}\beta+\frac{\bar L K_1}{\rho(1-K_{\mathrm D})}<\beta\), so \(r_+(\lambda_{\mathrm{bd}})=\beta\) and \(r_-(\lambda_{\mathrm{bd}})=K_{\mathrm D}\beta+\frac{\bar L K_1}{\rho(1-K_{\mathrm D})}\);
    \item if \(\frac{\bar L K_1}{\rho}>\beta(1-K_{\mathrm D})^2\), then \(K_{\mathrm D}\beta+\frac{\bar L K_1}{\rho(1-K_{\mathrm D})}>\beta\), so \(r_+(\lambda_{\mathrm{bd}})=K_{\mathrm D}\beta+\frac{\bar L K_1}{\rho(1-K_{\mathrm D})}\) and \(r_-(\lambda_{\mathrm{bd}})=\beta\).
\end{enumerate}

Assume that \(\frac{\bar L K_1}{\rho}<\beta(1-K_{\mathrm D})^2\). Then \(K_{\mathrm D}\beta+\frac{\bar L K_1}{\rho(1-K_{\mathrm D})}<\beta\), so at \(\lambda=\lambda_{\mathrm{bd}}\) we have
\[
r_+(\lambda_{\mathrm{bd}})=\beta,\qquad r_-(\lambda_{\mathrm{bd}})=K_{\mathrm D}\beta+\frac{\bar L K_1}{\rho(1-K_{\mathrm D})},
\qquad
\frac{\lambda_{\mathrm{bd}}(\lambda_{\mathrm{bd}}-\beta)-\frac{\bar L K_1}{\rho}\beta}
{\lambda_{\mathrm{bd}}-\beta}
=K_{\mathrm D}\beta+\frac{\bar L K_1}{\rho(1-K_{\mathrm D})}
=r_-(\lambda_{\mathrm{bd}}).
\]
Therefore, by \eqref{eq:BT_cpm},
\( c_+(\lambda_{\mathrm{bd}})=0. \)
Also, \(c_+\) is continuous on \((\lambda_{\mathrm{ac}},\infty)\), and if \(c_+(\lambda)=0\) for some \(\lambda>\lambda_{\mathrm{ac}}\), then
\( \lambda(\lambda-\beta)-\frac{\bar L K_1}{\rho}\beta=r_-(\lambda)(\lambda-\beta).\)
Since \(r_-(\lambda)\) satisfies \eqref{eq:BT_char_eq}, substituting
\[
r_-(\lambda)
=\frac{\lambda(\lambda-\beta)-\frac{\bar L K_1}{\rho}\beta}{\lambda-\beta}
\]
into \eqref{eq:BT_char_eq} forces
\(\lambda=\lambda_{\mathrm{bd}}.\)
Hence \(c_+\) has a unique zero on \((\lambda_{\mathrm{ac}},\infty)\), namely \(\lambda_{\mathrm{bd}}\). Since \(c_+(\lambda)>0\) for all sufficiently large \(\lambda\) (this is immediate from \eqref{eq:BT_cpm}, because \(\lambda(\lambda-\beta)-\frac{\bar L K_1}{\rho}\beta\sim\lambda^2\), \(\lambda-\beta\sim\lambda\), and \(r_-(\lambda)=O(1)\)), it follows that
\(
c_+(\lambda)<0\)
for all \(\lambda\in(\lambda_{\mathrm{ac}},\lambda_{\mathrm{bd}}).\)

Fix any \(\lambda\in(\lambda_{\mathrm{ac}},\lambda_{\mathrm{bd}})\). Since \(r_+(\lambda)>r_-(\lambda)>0\), the first term in \eqref{eq:BT_pn_realroots} dominates for large \(n\). Because \(c_+(\lambda)<0\), we obtain
\(
\det(\lambda I_n-\mathcal B_{n-1})<0
\) for all sufficiently large $n$.
Since the largest real root of \(\det(\lambda I_n-\mathcal B_{n-1})\) is \(\rho(\mathcal B_{n-1})\), the inequality \(\det(\lambda I_n-\mathcal B_{n-1})<0\) implies
\(
\rho(\mathcal B_{n-1})>\lambda\)
 for all sufficiently large $n$.
Equivalently,
\(
\liminf_{T\to\infty}\rho(\mathcal B_T)\ge \lambda.
\)
Letting \(\lambda\nearrow\lambda_{\mathrm{bd}}\), we get
\(
\liminf_{T\to\infty}\rho(\mathcal B_T)\ge \lambda_{\mathrm{bd}}.
\)

On the other hand, Lemma~\ref{lem:BT_basic} gives
\[
\rho(\mathcal B_T)\le \beta+\frac{\bar LK_1}{\rho}\sum_{j=0}^{T-1}K_{\mathrm D}^j
\nearrow \beta+\frac{\bar L K_1}{\rho(1-K_{\mathrm D})}
=\lambda_{\mathrm{bd}}.
\]
Therefore
\(
\limsup_{T\to\infty}\rho(\mathcal B_T)\le \lambda_{\mathrm{bd}}.
\)
Combining the two bounds yields
\(
\rho(\mathcal B_T)\nearrow \lambda_{\mathrm{bd}}.
\)

\medskip
\noindent
\textbf{Asymptotics in the regime \(\frac{\bar L K_1}{\rho}>\beta(1- K_{\mathrm D})^2\).}
In this case we have \(K_{\mathrm D}\beta+\frac{\bar L K_1}{\rho(1-K_{\mathrm D})}>\beta\), so at \(\lambda=\lambda_{\mathrm{bd}}\),
\[
r_+(\lambda_{\mathrm{bd}})=K_{\mathrm D}\beta+\frac{\bar L K_1}{\rho(1-K_{\mathrm D})},\qquad r_-(\lambda_{\mathrm{bd}})=\beta,
\qquad
\frac{\lambda_{\mathrm{bd}}(\lambda_{\mathrm{bd}}-\beta)-\frac{\bar L K_1}{\rho}\beta}
{\lambda_{\mathrm{bd}}-\beta}
=K_{\mathrm D}\beta+\frac{\bar L K_1}{\rho(1-K_{\mathrm D})}
=r_+(\lambda_{\mathrm{bd}}).
\]
Hence, by \eqref{eq:BT_cpm},
$
c_-(\lambda_{\mathrm{bd}})=0,$
$
c_+(\lambda_{\mathrm{bd}})\neq 0.
$
As before, $c_+$ is continuous on \((\lambda_{\mathrm{ac}},\infty)\), and \(c_+(\lambda)=0\) would force \(\lambda=\lambda_{\mathrm{bd}}\), which is impossible here because at \(\lambda_{\mathrm{bd}}\) the cancellation occurs in \(c_-\), not in \(c_+\). Hence
\(
c_+(\lambda)\neq 0 \)
for all \(\lambda>\lambda_{\mathrm{ac}}.\)
Since \(c_+(\lambda)>0\) for all sufficiently large \(\lambda\), continuity implies
\(
c_+(\lambda)>0\) for all \(\lambda>\lambda_{\mathrm{ac}}.\)

Fix \(\lambda>\lambda_{\mathrm{ac}}\). Then \(\Delta(\lambda)>0\), thus \eqref{eq:BT_pn_realroots} applies and the dominant term has positive coefficient \(c_+(\lambda)>0\). Therefore
\(
\det(\lambda I_n-\mathcal B_{n-1})>0\)
for all sufficiently large \(n\).
This implies
\(
\rho(\mathcal B_{n-1})<\lambda\)
for all sufficiently large \(n,\) and hence
\(
\limsup_{T\to\infty}\rho(\mathcal B_T)\le \lambda.
\)
Letting \(\lambda\downarrow\lambda_{\mathrm{ac}}\), we conclude that
\(
\limsup_{T\to\infty}\rho(\mathcal B_T)\le \lambda_{\mathrm{ac}}.
\)

It remains to prove the matching lower bound. Fix any \(\lambda<\lambda_{\mathrm{ac}}\). Then, \(\Delta(\lambda)<0\), and thus the two roots of \eqref{eq:BT_char_eq} are complex conjugates:
\(
r_\pm(\lambda)=m(\lambda)e^{\pm i\theta(\lambda)},\) where \(
m(\lambda)>0,\) and \(\theta(\lambda)\in(0,\pi).\)
Therefore, the real sequence \(\det(\lambda I_n-\mathcal B_{n-1})\) admits the representation
\begin{equation}\label{eq:BT_trig_form}
\det(\lambda I_n-\mathcal B_{n-1})
=
m(\lambda)^{n-1}
\Big(A(\lambda)\cos((n-1)\theta(\lambda))
+B(\lambda)\sin((n-1)\theta(\lambda))\Big),
\end{equation}
for suitable real coefficients \(A(\lambda),B(\lambda)\) determined by \(\lambda-\beta\) and \(\lambda(\lambda-\beta)-\frac{\bar L K_1}{\rho}\beta\).
The coefficients pair \((A(\lambda),B(\lambda))\) cannot be \((0,0)\), otherwise \(\det(\lambda I_n-\mathcal B_{n-1})\equiv 0\) for all \(n\), which is impossible for a fixed \(\lambda\).

Hence the trigonometric factor in \eqref{eq:BT_trig_form} is a nontrivial sinusoid sampled along an arithmetic progression, and therefore takes negative values infinitely often. Consequently,
\(
\det(\lambda I_n-\mathcal B_{n-1})<0\)
for infinitely many \(n\).
For those \(n\), we have
$\rho(\mathcal B_{n-1})>\lambda.$
Since the sequence \((\rho(\mathcal B_T))_T\) is increasing (Lemma~\ref{lem:BT_basic}), it follows that in fact
$
\liminf_{T\to\infty}\rho(\mathcal B_T)\ge \lambda.
$
Letting \(\lambda\nearrow\lambda_{\mathrm{ac}}\), we obtain
$
\liminf_{T\to\infty}\rho(\mathcal B_T)\ge \lambda_{\mathrm{ac}}.
$
Combining with the upper bound gives
$
\rho(\mathcal B_T)\nearrow \lambda_{\mathrm{ac}}.
$

\medskip
\noindent
\textbf{The critical case \(\frac{\bar L K_1}{\rho}=\beta(1-K_{\mathrm D})^2\).}
In this case, it holds that
\(
\lambda_{\mathrm{ac}}=\lambda_{\mathrm{bd}},
\)
and the conclusion follows by continuity from either side of the phase threshold. This completes the proof.
\end{proof}

\section{Discounted Terminal Value Invariance under Minorization}\label{sect:average-cost-contraction}

\begin{proof}[Proof of Theorem~\ref{thrm:avg}]
Fix $T\in\mathbb N$ and $\mu_{-T}\in\mathcal P(X)$. For
$\pmb Q=(Q_{-T},\ldots,Q_0)\in\mathcal C_{\alpha\beta}^{T+1}$, define
$\mu^{\pmb Q}_{-t+1}=H_2(Q_{-t},\mu^{\pmb Q}_{-t}),$
for $t=T,\ldots,1,$
with $\mu^{\pmb Q}_{-T}=\mu_{-T}$, and set
\begin{align*}
Q^{\pmb\mu^{\pmb Q}}_{-t}
&=H_1^{\mathrm{avg}}(Q_{-t+1},\mu^{\pmb Q}_{-t}),
&&t=T,\ldots,1,
\\
Q^{\pmb\mu^{\pmb Q}}_0
&=H_1^{\mathrm{avg}}(Q_0,\mu^{\pmb Q}_0).
\end{align*}
Let $\pmb Q,\widehat{\pmb Q}\in\mathcal C_{\alpha\beta}^{T+1}$. Since
$\mu^{\pmb Q}_{-T}=\mu^{\widehat{\pmb Q}}_{-T}$, repeated application of
\eqref{eq:quaso1} gives, for $t=T-1,\ldots,0$,
\begin{align}
\|\mu^{\pmb Q}_{-t}-\mu^{\widehat{\pmb Q}}_{-t}\|_1
&\leq
\frac{K_1}{\rho}W
\sum_{j=t+1}^{T}K_{\mathrm D}^{j-t-1}
\|Q_{-j}-\widehat Q_{-j}\|_\omega.
\label{eq:average-cost-forward-proof}
\end{align}
Hence, by \eqref{eq:quaso2} and \eqref{eq:average-cost-forward-proof},
\begin{align}
\|Q^{\pmb\mu^{\pmb Q}}_0
-Q^{\pmb\mu^{\widehat{\pmb Q}}}_0\|_\omega
&\leq
\alpha\beta\|Q_0-\widehat Q_0\|_\omega
+\frac{\bar L K_1}{\rho}W
\sum_{j=1}^{T}K_{\mathrm D}^{j-1}
\|Q_{-j}-\widehat Q_{-j}\|_\omega,
\label{eq:average-cost-terminal-row}
\\
\|Q^{\pmb\mu^{\pmb Q}}_{-t}
-Q^{\pmb\mu^{\widehat{\pmb Q}}}_{-t}\|_\omega
&\leq
\alpha\beta\|Q_{-t+1}-\widehat Q_{-t+1}\|_\omega
+\frac{\bar L K_1}{\rho}W
\sum_{j=t+1}^{T}K_{\mathrm D}^{j-t-1}
\|Q_{-j}-\widehat Q_{-j}\|_\omega
\label{eq:average-cost-interior-row}
\end{align}
for $t=1,\ldots,T$. Thus, after repeating Lemma \ref{lem:33333} in this setting, we obtain
\begin{equation}\label{eq:average-cost-matrix-proof}
\begin{bmatrix}
\|Q^{\pmb\mu^{\pmb Q}}_0-Q^{\pmb\mu^{\widehat{\pmb Q}}}_0\|_\omega\\
\|Q^{\pmb\mu^{\pmb Q}}_{-1}-Q^{\pmb\mu^{\widehat{\pmb Q}}}_{-1}\|_\omega\\
\vdots\\
\|Q^{\pmb\mu^{\pmb Q}}_{-T}-Q^{\pmb\mu^{\widehat{\pmb Q}}}_{-T}\|_\omega
\end{bmatrix}
\leq
\begin{bmatrix}
\alpha\beta & \frac{\bar L K_1}{\rho}W
& \frac{\bar LK_{\mathrm D} K_1}{\rho}W
& \cdots
& \frac{\bar LK_{\mathrm D}^{T-1}K_1}{\rho}W
\\
\alpha\beta & 0
& \frac{\bar L K_1}{\rho}W
& \cdots
& \frac{\bar LK_{\mathrm D}^{T-2}K_1}{\rho}W
\\
0&\alpha\beta&0&\cdots&\frac{\bar LK_{\mathrm D}^{T-3}K_1}{\rho}W
\\
\vdots&\vdots&\vdots&\ddots&\vdots
\\
0&0&0&\alpha\beta&0
\end{bmatrix}
\begin{bmatrix}
\|Q_0-\widehat Q_0\|_\omega\\
\|Q_{-1}-\widehat Q_{-1}\|_\omega\\
\vdots\\
\|Q_{-T}-\widehat Q_{-T}\|_\omega
\end{bmatrix}.
\end{equation}

The matrix in \eqref{eq:average-cost-matrix-proof} is obtained from the matrix
$\mathcal B_T$ in \eqref{eq:123455} by replacing
$\left(\beta, \frac{\bar L K_1}{\rho}\right)$ with $\left( \alpha \beta , \frac{\bar L K_1}{\rho}W\right)$.
The first condition in Theorem~\ref{thrm:avg} is equivalent to
\begin{equation}\label{eq:average-cost-square-condition}
\alpha\beta\left(
\sqrt{K_{\mathrm D}+\frac{\bar L K_1}{\rho}W}
+\sqrt{\frac{\bar L K_1}{\rho}W}
\right)^2<1.
\end{equation}
Moreover,
\begin{align*}
\alpha\beta K_{\mathrm D}+2\sqrt{\alpha\beta\frac{\bar L K_1}{\rho}W}
&=\alpha\beta\left(
\sqrt{K_{\mathrm D}+\frac{\bar L K_1}{\rho}W}
+\sqrt{\frac{\bar L K_1}{\rho}W}
\right)^2
\\&\quad
+2\sqrt{\alpha\beta\frac{\bar L K_1}{\rho}W}
\left[1-\sqrt{\alpha\beta}\left(
\sqrt{K_{\mathrm D}+\frac{\bar L K_1}{\rho}W}
+\sqrt{\frac{\bar L K_1}{\rho}W}
\right)\right]
\\
&\leq
\sqrt{\alpha\beta}\left(
\sqrt{K_{\mathrm D}+\frac{\bar L K_1}{\rho}W}
+\sqrt{\frac{\bar L K_1}{\rho}W}
\right)<1,
\end{align*}
where the inequality follows from
\[
2\sqrt{\alpha\beta\frac{\bar L K_1}{\rho}W}
\leq
\sqrt{\alpha\beta}\left(
\sqrt{K_{\mathrm D}+\frac{\bar L K_1}{\rho}W}
+\sqrt{\frac{\bar L K_1}{\rho}W}
\right).
\]
After the substitutions above, the strict phase condition places the matrix
in the second branch of Theorem~\ref{thm:BT_asymptotics}, regardless of whether
$K_{\mathrm D}<1$ or $K_{\mathrm D}\geq1$.  Therefore Theorem \ref{thm:BT_asymptotics} gives
\begin{equation}\label{eq:average-cost-perron-proof}
\rho(\mathcal B_T)\nearrow
\alpha\beta K_{\mathrm D}+2\sqrt{\alpha\beta\frac{\bar L K_1}{\rho}W}<1
\end{equation}
under
\[
\frac{\bar L K_1}{\rho}W>\alpha\beta(1-K_{\mathrm D})^2.
\]

By the Perron--Frobenius Theorem, there exists a vector with strictly positive
entries such that the transpose of the matrix in
\eqref{eq:average-cost-matrix-proof}, multiplied by this vector, equals
$\rho(\mathcal B_T)$ times the same vector. The scalar product with this vector
therefore defines a weighted norm under which the mapping
$\pmb Q\mapsto\pmb Q^{\pmb\mu^{\pmb Q}}$ is contractive. Since
$\mathcal C_{\alpha\beta}^{T+1}$ is invariant under this mapping, the Banach fixed-point
theorem yields a unique fixed point. Its forward state-measure flow and the
policies induced by its $Q$-functions give the unique finite-horizon MFE satisfying
$Q_0=H_1^{\mathrm{avg}}(Q_0,\mu_0)$. The iterates converge to this fixed point,
which proves the result.
\end{proof}

\end{document}